\documentclass[12pt, reqno]{amsart}
\usepackage{amsmath, amsthm, amscd, amsfonts, amssymb, graphicx, xcolor}
\usepackage[T1]{fontenc}
\usepackage[utf8]{inputenc} 
\usepackage[bookmarksnumbered, colorlinks, plainpages]{hyperref}
\usepackage{tikz,wrapfig}

\newtheorem{theorem}{Theorem}[section]
\newtheorem{lemma}[theorem]{Lemma}
\newtheorem{proposition}[theorem]{Proposition}
\newtheorem{corollary}[theorem]{Corollary}
\theoremstyle{definition}

\newtheorem{example}[theorem]{Example}

\theoremstyle{remark}
\newtheorem{remark}[theorem]{Remark}
\numberwithin{equation}{section}

\theoremstyle{plain}
\newtheorem*{theorem*}{Theorem}
\newtheorem*{lemma*}{Lemma}

\newcommand{\hyper}{\mathcal{H}}
\newcommand{\finite}{\mathfrak{F}}
\newcommand{\single}{\finite_1}
\newcommand{\diam}{\operatorname{diam}}
\newcommand{\Diam}{\operatorname{Diam}}
\newcommand{\ball}{\mathbf{B}}
\newcommand{\bounded}{\mathcal{B}}
\newcommand{\average}{\mathbf{A}}
\newcommand{\upperAdens}{\operatorname{ua-dens}}
\newcommand{\lowerAdens}{\operatorname{la-dens}}
\newcommand{\upperBdens}{\operatorname{ub-dens}}
\newcommand{\lcsigma}{$\operatorname{lc}_\sigma$ }

\newcommand{\haar}[1]{|#1|}
\newcommand{\dd}[1]{ \, \mathrm{d}#1}
\newcommand{\hyperhaar}{\hyper'}

\begin{document}
\setcounter{page}{1}

\color{darkgray}


\title[Mean Diameter, Regularity and Diam-Mean Equicontinuity]
{Mean Diameter, Regularity and Diam-Mean Equicontinuity}

\author[]{Till Hauser}
\address{Facultad de Matem\'aticas, Pontificia Universidad Cat\'olica de Chile. Edificio Rolando Chuaqui, Campus San Joaquín. Avda. Vicuña Mackenna 4860, Macul, Chile.}
\email{hauser.math@mail.de}

\thanks{This article was funded by the Deutsche Forschungsgemeinschaft (DFG, German Research Foundation) – 530703788. 
The author is grateful to María Isabel Cortez for her support and insightful advice throughout the development of this paper, and to the anonymous referee for a thorough and insightful report that helped improve the manuscript.
}
\begin{abstract}
    In the context of (not necessarily minimal) actions, we consider the mean diameter and use it to characterize regular factor maps. 
    Building on this characterization, we prove that an action is diam-mean equicontinuous if and only if it is a regular extension of its maximal equicontinuous factor. 
    Furthermore, we establish the existence of a maximal diam-mean equicontinuous factor and discuss stability properties of regular factor maps. 
    For this, we work in the context of actions of locally compact and $\sigma$-compact amenable groups. 
\newline

 \end{abstract} \maketitle


\section{Introduction}

Equicontinuous actions are fundamental building blocks of the theory of topological dynamics. 
To study more intricate systems, researchers often turn to weaker derivatives of equicontinuity, such as distality \cite{auslander1988minimal}, mean equicontinuity \cite{fomin1951dynamical, li2015mean, fuhrmann2022structure}, and diam-mean equicontinuity \cite{garcia2017weak, garcia2021mean}, among others.

A central technique in the field is to relate a given dynamical system to an equicontinuous one, typically via the maximal equicontinuous factor. A natural and important question is how dynamical properties of the original system are reflected in the structure of the corresponding factor map.

For instance, a landmark result due to Furstenberg states that a minimal action is distal if and only if the factor map to its maximal equicontinuous factor can be expressed as a (possibly transfinite) composition of equicontinuous factor maps \cite{auslander1988minimal}.
Another elegant result concerns mean equicontinuity: an action is mean equicontinuous if and only if the factor map to its maximal equicontinuous factor is topo-isomorphic. See \cite{downarowicz2016isomorphic, fuhrmann2022structure} for further details.

In this article, we focus on diam-mean equicontinuous actions and their relation to regular factor maps \cite{garcia2017weak, garcia2021mean, haupt2025multivariate, haupt2025note}.  
For this, we will work in the context of actions of amenable locally compact and $\sigma$-compact topological groups, which we abbreviate as \emph{amenable \lcsigma groups}. 
This is motivated by the natural appearance of actions of amenable \lcsigma groups, such as $\mathbb{R}^n$, in the context of aperiodic order \cite{baake2013aperiodic, solomyak2020delone} and the study of regular Toeplitz actions for countable amenable (residually finite) groups \cite{downarowicz2005survey, baake2016toeplitz, cortez2014invariant}. 


Our central object of study is the mean diameter. 
To define it let $X$ be an action of an amenable \lcsigma group $G$ on a compact metric space and denote by $d$ the metric of $X$. We fix a (left) Haar measure on $G$ and denote it by $\haar{\cdot}.$ 
A sequence $(F_n)$ of compact subsets of $G$ with $\haar{F_n}>0$ is called \emph{left F\o lner} if for every non-empty compact $K\subseteq G$ we have $\haar{F_n\Delta KF_n}/\haar{F_n}\to 0$. 
For a left F\o lner sequence $\mathcal{F}=(F_n)$ and a closed subset $A\subseteq X$ we define the \emph{mean diameter of $A$ w.r.t.\ $\mathcal{F}$} as
\begin{align*}
    \Diam_\mathcal{F}(A):=\limsup_{n\to \infty} \frac{1}{\haar{F_n}}\int_{F_n}\diam(g.A)\dd g,
\end{align*}
where we abbreviate $g.A:=\{g.x;\, x\in A\}$ and the \emph{diameter} of $A$ is given by 
$\diam(A):=\max_{x,x'\in A}d(x,x')$. 
The \emph{mean diameter} of $A$ is defined by 
\[\Diam(A):=\sup_{\mathcal{F}}\Diam_\mathcal{F}(A),\] 
where the supremum is taken over all left F\o lner sequences $\mathcal{F}$ of $G$. 
In Section~\ref{sec:meanDiameter}, we present basic properties of $\Diam$ together with several methods for its computation, such as 
    \[
    \Diam(A)=\inf_K \sup_{h\in G} \frac{1}{\haar{Kh}}\int_{Kh}\diam(g.A) \dd g,\]
where the infimum is taken over all compact subsets $K$ of $G$ with $\haar{K}>0$. 

We will use the mean diameter to characterize regular factor maps. 
For this we call a factor map $\pi\colon X\to Y$ \emph{diam-mean proximal} if for all $y\in Y$ we have $\Diam(\pi^{-1}(y))=0$.
Recall that a factor map $\pi\colon X\to Y$ is called \emph{regular} (also \emph{almost surely one-to-one}) if the set of injectivity points $X_0$ of $\pi$ satisfies $\mu(X_0)=1$ for all invariant Borel probability measures on $X$. 
See Section \ref{sec:regularityVSmeanDiamProximality} for further details. 
In Theorem \ref{the:regularFactorMapCharacterization} we establish the following characterization of regularity. 

\begin{theorem*}
    A factor map is regular if and only if it is diam-mean proximal. 
\end{theorem*}

It is also natural to study the following notion. 
For a left F\o lner sequence $\mathcal{F}$ in $G$ we call a factor map $\pi\colon X\to Y$ \emph{$\mathcal{F}$-diam-mean proximal} if for all $y\in Y$ we have $\Diam_\mathcal{F}(\pi^{-1}(y))=0$.
Clearly any diam-mean proximal factor map is $\mathcal{F}$-diam-mean proximal. 
In Theorem \ref{the:regularFactorMapCharacterization} we will see that also the converse holds and show the following. 

\begin{theorem*}
    A factor map is diam-mean proximal if and only if it is $\mathcal{F}$-diam-mean proximal for some left F\o lner sequence $\mathcal{F}$ in $G$. 
\end{theorem*}

Another important application of the mean diameter is the concept of diam-mean equicontinuity \cite{garcia2017weak, garcia2021mean}. 
An action is called \emph{diam-mean equicontinuous} (also \emph{Banach diam-mean equicontinuous} \cite{garcia2021mean}) if for all $\epsilon>0$ and every $x\in X$, there exists a compact neighbourhood $A$ of $x$ with $\Diam(A)\leq \epsilon$. 
We will see in Section \ref{sec:diamMeanEquicontinuity} that this notion is independent of the choice of a metric on $X$ and provide some basic characterizations. 

Recall that diam-mean proximality and regularity are equivalent concepts for all factor maps, so it is to be expected that the concepts of diam-mean equicontinuity and regularity are strongly related. 
The first result of this interplay was achieved in \cite[Theorem~54]{garcia2017weak} in the context of minimal almost automorphic subshifts \cite[Definition~50]{garcia2017weak}. 
Generalizing the result of \cite{garcia2017weak} it was presented in \cite[Theorem 4.12]{garcia2021mean} that whenever $X$ is a minimal action of $\mathbb{Z}$, then $X$ is diam-mean equicontinuous if and only if the factor map $\pi_{\operatorname{eq}}\colon X\to X_{\operatorname{eq}}$ onto the maximal equicontinuous factor is regular. 
This result was later generalized to minimal actions of more general acting groups, such as Abelian \lcsigma groups \cite[Theorem~4.5]{haupt2025multivariate}. 
Recent progress was made in \cite[Theorem~5.2]{haupt2025note} for the context of actions of countable amenable groups, where it was shown that a minimal action that satisfies the local Bronstein property \cite[Definition~2.14]{haupt2025note} is diam-mean equicontinuous if and only if $\pi_{\operatorname{eq}}$ is regular. 

Nevertheless, the question remained whether the characterization holds for all minimal actions of countable amenable groups. 
Furthermore, it was asked in \cite[Question 2]{garcia2021mean}, whether this characterization also holds in the non-minimal context. 
In Section \ref{sec:regularityVSdiamMeanEquicontinuity} we will give an affirmative answer to both questions and show the following. 

\begin{theorem}
\label{the:INTROcharacterizationDMEvsRegular}
    An action of an amenable \lcsigma group on a compact metric space is diam-mean equicontinuous if and only if the factor map onto the maximal equicontinuous factor is regular. 
\end{theorem}

Note that our proof significantly differs from \cite{garcia2021mean, haupt2025multivariate, haupt2025note} and does not involve the concept of frequent stability \cite[Section 3]{garcia2021mean}. 

Motivated by Theorem \ref{the:INTROcharacterizationDMEvsRegular} we will explore stability properties of regular factor maps in Section \ref{sec:stabilityRegularFactorMaps}.
This will allow us to prove that diam-mean equicontinuity is stable under (countable) products and to establish the following in Section \ref{sec:maximalDiamMeanEquicontinuousFactor}. 

\begin{theorem}[Maximal diam-mean equicontinuous factor]
\label{the:INTROmaximalDiamMeanEquicontinuousFactor}
    Let $X$ be an action of an amenable \lcsigma group $G$ on a compact metric space. 
    
    There exists a maximal diam-mean equicontinuous factor of $X$, i.e.\ a factor map $\pi_{\operatorname{dm}}\colon X\to X_{\operatorname{dm}}$ onto a diam-mean equicontinuous action $X_{\operatorname{dm}}$ of $G$ such that for any factor map $\pi\colon X\to Y$ onto a diam-mean equicontinuous action of $G$ there exists a factor map $\psi\colon X_{\operatorname{dm}}\to Y$ such that $\pi=\psi\circ \pi_{\operatorname{dm}}$. 
    The maximal diam-mean equicontinuous factor of $X$ is unique up to conjugacy. 
\end{theorem}

\subsection*{Convention}

Throughout this article, $G$ denotes an amenable \lcsigma group, unless explicitly stated otherwise. 
We fix a (left) Haar measure on $G$, denoted by $\haar{\cdot}$.
For readability, we often omit explicit mention of the group $G$, and refer simply to an action, a factor map, or a Følner sequence. Unless stated otherwise, it is understood that all such objects refer to the same fixed group $G$.

\section{Preliminaries}
\label{sec:prelims}

We denote $\mathbb{R}^+:=[0,\infty)$. 
If $X$ is a topological space, we denote by $\single(X)$ and $\hyper(X)$ the sets of all singletons and non-empty compact subsets, respectively.
Whenever we speak of \emph{measurability}, we refer to the Borel $\sigma$-algebra.
We write $\hyperhaar(G)$ for the set of all compact subsets $K$ of $G$ with $\haar{K}>0$.

\subsection{Hyperspaces}
Let $(X,d)$ be a compact metric space. 
We write 
$\ball_\epsilon^d(x):=\{x'\in X;\, d(x,x')\leq \epsilon\}$ 
and 
$B_\epsilon^d(x):=\{x'\in X;\, d(x,x')< \epsilon\}$
for $\epsilon>0$ and $x\in X$. 
For $A\in \hyper(X)$ we denote $\diam_d(A)=\max_{x,x'\in A}d(x,x')$. 
If the metric is understood from the context, we simply write $\ball_\epsilon(x)$, $B_\epsilon(x)$ and $\diam(A)$. 
For $A\in \hyper(X)$ and $\epsilon>0$ we denote $\ball_\epsilon(A):=\bigcup_{x\in A}\ball_\epsilon(x)$. 
The \emph{Hausdorff metric} is given by 
\begin{align*}
    d_\hyper(A,A'):=\min\{\epsilon>0;\, A\subseteq \ball_\epsilon(A') \text{~~~and~~~}A'\subseteq \ball_\epsilon(A)\},
\end{align*}
for $A,A'\in \hyper(X)$.
Note that, as a straightforward consequence of the compactness of $X$, the infimum in the above formula is attained. 
It is well known that $d_\hyper$ is a metric on $\hyper(X)$ and that $\hyper(X)$ equipped with the Hausdorff metric $d_\hyper$ is a compact metric space. 
We recommend \cite{nadler1992continuum} for more details on hyperspaces and note that $\diam\colon \hyper(X)\to \mathbb{R}^+$ is continuous.

\subsection{Amenable groups and F\o lner sequences}
Let $G$ be a locally compact and $\sigma$-compact group (\emph{\lcsigma group}) and choose a (left) Haar measure $\haar{\cdot}$ on $G$. 
A sequence $(F_n)$ in $\hyperhaar(G)$ is called \emph{left F\o lner} if
$\haar{F_n\Delta KF_n}/\haar{F_n}\to 0$ 
holds for all $K\in \hyper(G)$, where $\Delta$ denotes the symmetric difference. 
It is called \emph{right F\o lner} whenever 
$\haar{F_n\Delta F_nK}/\haar{F_n}\to 0$ 
holds for all $K\in \hyper(G)$. 
A \lcsigma group that allows for a left F\o lner sequence is called \emph{amenable}.

\subsection{Actions}
\label{subsec:prelimsActions}
Let $G$ be a topological group and $X$ be a compact metric space. 
A \emph{(continuous) action} of $G$ on $X$ is a group homomorphism $\alpha$ from $G$ into the group of homeomorphisms $X\to X$ such that $G\times X\to X\colon (g,x)\mapsto \alpha(g)(x)$ is continuous. 
For $g\in G$ and $x\in X$, we suppress the symbol for the action by simply writing $g.x:=\alpha(g)(x)$. 
This allows us to simply speak of an \emph{action $X$ of $G$}. 
Whenever the group $G$ is given by the context, we simply speak of an action $X$. 
See \cite{auslander1988minimal} for more details on actions.   

For $x\in X$ we denote $O(x):=\{g.x; g\in G\}$ for the \emph{orbit} of $x$. 
A point $x\in X$ is called \emph{transitive} if it has a dense orbit. 
An action is called \emph{transitive} if it allows for a transitive point. It is called \emph{minimal} if every point is transitive. 

A subset $M\subseteq X$ is called \emph{invariant} if $M=g.M:=\{g.x;\, x\in M\}$ holds for all $g\in G$. 
For an invariant $A\in \hyper(X)$ the action of $G$ on $X$ naturally induces an action of $G$ on $A$. We speak of a \emph{subaction (of $X$)} in this case.
Whenever $(X_n)_{n\in \mathbb{N}}$ is a countable family of actions, then we consider the action of $G$ on $\prod_n X_n$ given by $g.(x_n)_n:=(g.x_n)_n$ for $g\in G$ and $(x_n)_n\in \prod_n X_n$.     
Furthermore, whenever $X$ is an action, then $(g,A)\mapsto g.A$ induces an action of $G$ on the hyperspace $\hyper(X)$. 
We simply write $\prod_n X_n$ and $\hyper(X)$ for the respective actions.

\subsection{Factor maps}
A continuous surjection $\pi\colon X\to Y$ between actions is called a \emph{factor map} (also \emph{factor of $X$}) if $\pi(g.x)=g.\pi(x)$ holds for all $g\in G$ and $x\in X$. A factor map is called a \emph{conjugacy} if it is a homeomorphism. 
Note that any factor map is \emph{closed}, i.e.\ that $\pi(A)$ is closed for all closed $A\subseteq X$. 

An action $X$ is called \emph{equicontinuous} if there exists a continuous metric $d$ on $X$ such that $d(g.x,g.x')=d(x,x')$ for all $x,x'\in X$. 
For any action $X$ there exists a \emph{maximal equicontinuous factor} of $X$, i.e.\ a factor map $\pi_{\operatorname{eq}}\colon X\to X_{\operatorname{eq}}$ onto an equicontinuous action $X_{\operatorname{eq}}$ such that for any factor map $\pi\colon X\to Y$ onto an equicontinuous action there exists a factor map $\psi\colon X_{\operatorname{eq}}\to Y$ with $\pi=\psi \circ \pi_{\operatorname{eq}}$. For details see \cite{auslander1988minimal}.

\subsection{Invariant measures}
We denote by $\mathcal{M}(X)$ the set of all Borel probability measures on $X$.
Equipped with the weak*-topology $\mathcal{M}(X)$ becomes a compact metric space \cite{glasner2003ergodic, villani2003topics, fuhrmann2025continuity}. 
For $\mu\in \mathcal{M}(X)$ and a continuous surjection $\pi\colon X\to Y$ we denote $\pi_*\mu$ for the \emph{pushforward measure}, given by 
$\pi_*\mu(M):=\mu(\pi^{-1}(M))$ for all measurable $M\subseteq Y$.
Clearly, we have $\pi_*\mu\in \mathcal{M}(Y)$. 

For $g\in G$ and $\mu \in \mathcal{M}(X)$ we denote $g_*\mu$ for the pushforward under the homeomorphism induced by $g$. 
A measure $\mu\in \mathcal{M}(X)$ is called \emph{invariant} if $g_*\mu=\mu$ holds for all $g\in G$. 
We denote by $\mathcal{M}_G(X)$ the set of all invariant $\mu\in \mathcal{M}(X)$. 
An action is called \emph{uniquely ergodic} if it allows for a unique invariant probability measure. 
An invariant measure $\mu\in \mathcal{M}(X)$ is called \emph{ergodic} if for each invariant measurable $M\subseteq X$ we have that $\mu(M)\in \{0,1\}$. 
Let $\pi\colon X\to Y$ be a factor map. 
Note that for $\mu\in \mathcal{M}_G(X)$ we have $\pi_*\mu\in \mathcal{M}_G(Y)$ and that ergodicity is preserved under pushing forward. 

\subsection{Actions of amenable \lcsigma groups}
Let $X$ be an action of an amenable \lcsigma~group $G$. 
For $x\in X$, we denote by $\delta_x$ the Dirac measure at $x$. 
For $K\in \hyperhaar(G)$ and $\mu\in \mathcal{M}(X)$ we denote 
$K_*\mu := {1}/{\haar{K}}\int_K (g^*\mu) \dd g$, i.e.\ 
\[
(K_*\mu)(M)
:=\frac{1}{\haar{K}}\int_K \mu(g^{-1}.M)\dd g
\] 
for all $M\subseteq X$ measurable. 
Note that $K_*\mu\in \mathcal{M}(X)$. 
For a left (or right) F\o lner sequence $\mathcal{F}=(F_n)$ in $G$ and $\mu\in \mathcal{M}(X)$ we say that $x$ is \emph{$\mathcal{F}$-generic for $\mu$} if $((F_n)_*\delta_x)_{n\in \mathbb{N}}$ converges to $\mu$ in $\mathcal{M}(X)$. 
We say that $x$ is \emph{$\mathcal{F}$-generic} if it is $\mathcal{F}$-generic for some $\mu\in \mathcal{M}(X)$.     
Note that for $x\in X$ any left F\o lner sequence $\mathcal{F}$ allows for a subsequence ${\mathcal{F}}'$ such that $x$ is {${\mathcal{F}}'$-generic}. 
Furthermore, from \cite[Theorem 2.4]{fuhrmann2022structure} we know that for an ergodic measure $\mu\in \mathcal{M}_G(X)$ and any left F\o lner sequence $\mathcal{F}$ there exists a subsequence $\mathcal{F}'$ such that $\mu$-almost every point is $\mathcal{F}$-generic for $\mu$. 

It follows from a standard Krylov-Bogolyubov argument that for a left F\o lner sequence $\mathcal{F}=(F_n)$ and $x\in X$ any cluster point of $((F_n)_*\delta_x)$ is invariant.
Thus, any point in a uniquely ergodic action is $\mathcal{F}$-generic w.r.t.\ all left F\o lner sequences $\mathcal{F}$. 
Furthermore, from the sequential compactness of $\mathcal{M}(X)$ it is straightforward to observe that $\mathcal{M}_G(X)\neq \emptyset$.
Another straightforward argument yields that for a factor map $\pi\colon X\to Y$ and an ergodic $\nu\in \mathcal{M}_G(Y)$ there exists $\mu\in \mathcal{M}_G(X)$ with $\pi_*\mu=\nu$. 
These observations rely on the amenability of $G$.

\subsection{Subadditivity techniques}
\label{subsec:subadditivity}
A function $H\colon \hyperhaar(G)\to \mathbb{R}^+$ is called \emph{right dominated} if for all $K\in \hyperhaar(G)$ we have $\sup_{g\in G}H(Kg)/\haar{Kg}<\infty$. 
It is called 
\emph{right measurable} if the function $G\to \mathbb{R}^+$ defined by $g\mapsto H(Kg)$ is measurable for all $K\in \hyperhaar(G)$.
A right measurable function $H\colon\hyperhaar(G)\to \mathbb{R}^+$ is said to satisfy~\ref{S}, if for all $K,F\in \hyperhaar(G)$ we have 
\begin{align}
    \label{S}
        H(F)\leq \frac{1}{\haar{K}}\int_{K^{-1}F} H(Kg)\dd g. \tag{$\mathfrak{S}$}
\end{align}
Note that the integral in the above formula is well defined, since $g\mapsto H(Kg)\in \mathbb{R}^+$ is measurable. 

\begin{remark}
\label{rem:shearer}
    For countable discrete amenable groups $G$, Shearer's inequality \cite[Definition~2.1]{downarowicz2016shearers} implies \ref{S}. This implication follows from the arguments presented in the proof of \cite[Proposition 3.3]{downarowicz2016shearers}. However, since we are not aware of an appropriate formulation of Shearer's inequality for non-discrete amenable groups, we introduce \ref{S}, which is better suited to our setting.
    Note that \ref{S} serves as an intermediate condition between Shearer's inequality and the so-called 'infimum rule' \cite[Definition 3.1]{downarowicz2016shearers}, which was implicitly employed in the proof of \cite[Proposition 3.3]{downarowicz2016shearers}.
\end{remark}

\begin{example}
\label{exa:SpropertyForGfunctions}
    For bounded and measurable $f\colon G\to \mathbb{R}^+$ the function 
    \[H\colon \hyperhaar(G)\to \mathbb{R}^+\colon K\mapsto\int_K f(g) \dd g \]
    is right dominated, right measurable and satisfies~\ref{S}. 
\end{example}
\begin{proof}
    Since compact subsets of $G$ have finite Haar measure and $f$ is bounded and measurable,
    the integral defining $H$ is well-defined and finite.
    Furthermore, the boundedness of $f$ yields 
    \[\sup_{g\in G}\frac{H(Kg)}{\haar{Kg}}=\sup_{g\in G} \frac{1}{\haar{Kg}}\int_{Kg} f(h)\dd h \leq \|f\|_\infty,\]
    for all $K\in \hyperhaar(G)$, i.e.\ that $H$ is right dominated. 
    From the continuity of the multiplication $G\times G\to G$ we know that $(h,g)\mapsto f(hg)$ is measurable. 
    Using the Fubini-Tonelli theorem we observe that $H$ is right measurable and compute the following for $K,F\in \hyperhaar(G)$. 
    \begin{align*}
        \int_{K^{-1}F} H(Kg)\dd g 
        &=\int_{K^{-1}F} \int_{K} f(hg)\dd h\dd g 
        =\int_{K}\int_{K^{-1}F} f(hg)\dd g \dd h\\
        &\geq \int_{K}\int_{h^{-1}F} f(hg)\dd g\dd h
        = \haar{K}\int_{F} f(g) \dd g
        = \haar{K}H(F).
    \end{align*}
    This verifies that $H$ satisfies~\ref{S}. 
\end{proof}

\begin{proposition}
\label{pro:subadditivityGeneral}
    Let $H\colon \hyperhaar(G)\to \mathbb{R}^+$ be a right dominated and right measurable function that satisfies~\ref{S}, and $(F_n')$ a F\o lner sequence in $G$. We have 
    \begin{align*}
        \lim_{n\to \infty} \sup_{g\in G} \frac{H(F_n'g)}{\haar{F_n'g}}
        =
        \inf_{K\in \hyperhaar(G)}\sup_{g\in G} \frac{H(Kg)}{\haar{Kg}}
        &=
        \sup_{(F_n)} \limsup_{n\to \infty} \frac{H(F_n)}{\haar{F_n}} <\infty,
    \end{align*}
    where the supremum, taken over all F\o lner sequences $(F_n)$, is attained. 
    Moreover, the same statement holds with $\limsup$ replaced by $\liminf$.
\end{proposition}
\begin{remark}
\label{rem:modularFunction}
    $G$ is called \emph{unimodular} if $\haar{\cdot}$ is additionally right invariant. 
    Note that Abelian \lcsigma~groups, as well as discrete groups are unimodular. 
    If $G$ is a unimodular amenable \lcsigma~group, then the above formulas can be simplified by rewriting $H(Kg)/\haar{Kg}=H(Kg)/\haar{K}$ for $K\in \hyperhaar(G)$. 
    
    The \emph{modular function} $\chi_G\colon G\to (0,\infty)$ is defined by $\chi_G(g):=\haar{Kg}/\haar{K}$. We regard $(0,\infty)$ as a multiplicative group. 
    It is known that $\chi_G$ is a continuous group homomorphism and that it is independent of the choice of $K\in \hyperhaar(G)$ \cite[Chapter~1]{deitmar2014principles}. 
    For non-unimodular groups, $\chi_G \not \equiv 1$, and hence $\chi_G(G)$ is a non-trivial subgroup of $(0,\infty)$. 
    It follows that $\chi_G(G)$ is unbounded.  
    For $H$ defined by $H(K):=\haar{K}$ we know from Example~\ref{exa:SpropertyForGfunctions} (with $f\equiv 1$) that $H$ is right dominated, right measurable and satisfies \ref{S}. 
    Nevertheless, for any $K\in \hyperhaar(G)$ we have 
    $
    \sup_{g\in G}{H(Kg)}/{\haar{Kg}}
    =
    1
    <
    \infty 
    =
    \sup_{g\in G}\chi_G(g)
    =
    \sup_{g\in G}{H(Kg)}/{\haar{K}}.$
    Thus, the discussed simplification is only possible in the context of unimodular groups. 
\end{remark}

\begin{remark} 
    A function $H\colon \hyperhaar(G)\to \mathbb{R}^+$ is called \emph{right invariant} if $H(Kg)=H(K)$ holds for all $g\in G$. 
    If $G$ is unimodular and $H$ is right invariant and satisfies \ref{S}, then $H$ is right dominated and right measurable, and Proposition~\ref{pro:subadditivityGeneral} yields that for any F\o lner sequence $(F_n)$ in $G$ we have 
    \[\frac{H(F_n)}{\haar{F_n}}
    \to 
    \inf_{K\in \hyperhaar(G)}\frac{H(K)}{\haar{K}}.
    \]
    Recall from Remark~\ref{rem:shearer} that, for countable discrete groups, Shearer's inequality implies \ref{S}.
    Thus, Proposition~\ref{pro:subadditivityGeneral} extends \cite[Proposition 3.3]{downarowicz2016shearers}, from which the present result was inspired.
\end{remark}

\begin{proof}[Proof of Proposition \ref{pro:subadditivityGeneral}:]
    Since $H$ is right dominated, for each $n\in \mathbb{N}$ we have 
    $
    \sup_{g\in G} {H(F_n'g)}/{\haar{F_n'g}}
    <
    \infty
    $ 
    and find $g_n\in G$ with
    \begin{align*}
        \sup_{g\in G} \frac{H(F_n'g)}{\haar{F_n'g}}\leq \frac{H(F_n'g_n)}{\haar{F_n'g_n}}+\frac{1}{n}. 
    \end{align*}
    Since $(F_n'g_n)$ is a left F\o lner sequence and $F_n'\in \hyperhaar(G)$ for every $n$, we have 
    \begin{align*}
        \inf_{K\in \hyperhaar(G)}\sup_{g\in G} \frac{H(Kg)}{\haar{Kg}}
        &\leq 
        \liminf_{n\to \infty} \sup_{g\in G} \frac{H(F_n'g)}{\haar{F_n'g}}
        \leq 
        \limsup_{n\to \infty} \sup_{g\in G} \frac{H(F_n'g)}{\haar{F_n'g}}\\
        &\leq 
        \limsup_{n\to \infty} \frac{H(F_n'g_n)}{\haar{F_n'g_n}}
        \leq 
        \sup_{(F_n)} \limsup_{n\to \infty} \frac{H(F_n)}{\haar{F_n}},
    \end{align*}
    where the supremum is taken over all F\o lner sequences $(F_n)$.
    Similarly, we observe 
    \begin{align*}
        \inf_{K\in \hyperhaar(G)}\sup_{g\in G} \frac{H(Kg)}{\haar{Kg}}
        \leq 
        \sup_{(F_n)} \liminf_{n\to \infty} \frac{H(F_n)}{\haar{F_n}}
        \leq 
        \sup_{(F_n)} \limsup_{n\to \infty} \frac{H(F_n)}{\haar{F_n}}.
    \end{align*}
    Since $H$ is right dominated we have 
    $
    \inf_{K\in \hyperhaar(G)}\sup_{g\in G} {H(Kg)}/{\haar{Kg}}<\infty. 
    $
    Thus, it remains to show that for every F\o lner sequence $(F_n)$ we have 
    \begin{align}
    \label{equ:DFRargument}
    \tag{$\spadesuit$}
        \limsup_{n\to \infty} \frac{H(F_n)}{\haar{F_n}}
        &\leq 
        \inf_{K\in \hyperhaar(G)}\sup_{g\in G} \frac{H(Kg)}{\haar{Kg}}.    
    \end{align}
    For this consider $K\in \hyperhaar(G)$ and $\epsilon>0$.
    Recall from Remark~\ref{rem:modularFunction} that $\chi_G$ denotes the modular function.
    Since $K^{-1}F_n$ is compact and $\chi_G$ is continuous, 
    we find $h_n\in K^{-1}F_n$ with 
    $\chi_G(h_n)=\max_{g\in K^{-1}F_n}\chi_G(g)$. 
    For $g\in (Kh_n)^{-1}F_n=h_n^{-1}K^{-1}F_n$ we have $h_ng\in K^{-1}F_n$ and hence  
    \begin{align*}        \chi_G(g)=\chi_G(h_n^{-1})\chi_G(h_n g)\leq \chi_G(h_n^{-1})\chi_G(h_n)= 1. 
    \end{align*}
    Thus, from \ref{S} we observe
    \begin{align*}
        H(F_n)
        &\leq 
        \frac{1}{\haar{Kh_n}}\int_{(Kh_n)^{-1}F_n} H(Kh_ng) \dd g
        =
        \int_{(Kh_n)^{-1}F_n}
        \frac{\haar{Kh_ng}}{\haar{Kh_n}}
        \frac{H(Kh_ng)}{\haar{Kh_ng}} \dd g\\
        &\leq 
        \left(\int_{(Kh_n)^{-1}F_n}\chi_G(g) \dd g\right)
        \left(\sup_{g\in G} \frac{H(Kg)}{\haar{Kg}}\right)\\
        &\leq
        \haar{(Kh_n)^{-1}F_n}
        \left(\sup_{g\in G} \frac{H(Kg)}{\haar{Kg}}\right).     
    \end{align*}
    For large $n$ we have $\haar{K^{-1}F_n\Delta F_n}\leq \epsilon \haar{F_n}$ and hence  
    \begin{align*}
        \haar{(Kh_n)^{-1}F_n}
        =
        \haar{h_n^{-1}K^{-1}F_n}
        = 
        \haar{K^{-1}F_n}\leq \haar{K^{-1}F_n\Delta F_n}+\haar{F_n}\leq (1+\epsilon)\haar{F_n}.     
    \end{align*}
    This shows 
    \begin{align*}
        H(F_n)
        \leq
        \haar{(Kh_n)^{-1}F_n}
        \left(\sup_{g\in G} \frac{H(Kg)}{\haar{Kg}}\right)
        \leq 
        (1+\epsilon)\haar{F_n}
        \left(\sup_{g\in G} \frac{H(Kg)}{\haar{Kg}}\right)
    \end{align*}
    for large $n$ and hence
    \begin{align*}
        \limsup_{n\to \infty} \frac{H(F_n)}{\haar{F_n}}
        \leq (1+\epsilon)
        \sup_{g\in G} \frac{H(Kg)}{\haar{Kg}}.
    \end{align*}
    The claim \eqref{equ:DFRargument} follows from the arbitrariness of $\epsilon>0$. 
\end{proof}

\subsection{Upper densities}
Let $G$ be an amenable \lcsigma group. 
Consider a measurable subset $G'\subseteq G$. 
For a left (or right) F\o lner sequence $\mathcal{F}=(F_n)$ we define the 
\emph{upper asymptotic density} as $\upperAdens_\mathcal{F}(G'):=\limsup_{n\to \infty}\haar{G'\cap F_n}/\haar{F_n}$ and the 
\emph{lower asymptotic density} as $\lowerAdens_\mathcal{F}(G'):=\liminf_{n\to \infty}\haar{G'\cap F_n}/\haar{F_n}$. 
Furthermore, we define the \emph{upper Banach density} as 
$\upperBdens(G'):=\sup_{\mathcal{F}}\upperAdens_{\mathcal{F}}(G')$, where the supremum is taken over all left F\o lner sequences $\mathcal{F}$ in $G$. 
The following lemma shows that our definition agrees with the definition in the literature. See \cite[Section 2.2]{downarowicz2019tilings} and the references therein for more details in the context of actions of countable amenable groups. 

\begin{lemma}    \label{lem:upperBanachDensityCharacterization}
    For any measurable $G'\subseteq G$ we have 
    \[
        \upperBdens(G')
        =\sup_{\mathcal{F}}\upperAdens_{\mathcal{F}}(G')
        =\sup_{\mathcal{F}}\lowerAdens_{\mathcal{F}}(G') 
        =\inf_{K\in \hyperhaar(G)}\sup_{g\in G}\frac{\haar{G'\cap Kg}}{\haar{Kg}},
    \]
    where the suprema, taken over all left F\o lner sequences $\mathcal{F}$ in $G$, are attained. 
    Furthermore, for every left F\o lner sequence $(F_n)$, we have  
    \[
        \sup_{g\in G}\frac{\haar{G'\cap F_n g}}{\haar{F_ng}} \to \upperBdens(G').
    \]
\end{lemma}
\begin{proof}
    For $K\in \hyperhaar(G)$ denote $H(K):=\haar{G'\cap K}$. 
    Since the characteristic function $\mathbf{1}_{G'}\colon G\to \mathbb{R}^+$ is bounded, measurable and satisfies 
    $H(K)=\int_K \mathbf{1}_{G'}(g) \dd g$ we observe from Example~\ref{exa:SpropertyForGfunctions} that $H\colon \hyperhaar(G)\to \mathbb{R}^+$ is right dominated, right measurable and satisfies \ref{S}. 
    Thus, the statement follows from Proposition \ref{pro:subadditivityGeneral}. 
\end{proof}

\section{Averages of functions}
\label{sec:averages}

Throughout this section, let $X$ be an action of an amenable \lcsigma group $G$. 
We denote by $\bounded(X)$ the space of all bounded functions $X \to \mathbb{R}$, and by $C(X)$ the subspace of continuous functions.
We equip $\bounded(X)$ and $C(X)$ with the supremum norm and the pointwise ordering.
For $g\in G$ and $f\in \bounded(X)$, we denote $g^*f:=f\circ g$, 
where $g$ is identified with the homeomorphism $x\mapsto g.x$.
A function $f\in \bounded(X)$ is called \emph{invariant} if $g^*f=f$ for all $g\in G$. 
It is called \emph{positive} if $f(x)\geq 0$ holds for all $x\in X$. 
We denote by $\bounded_G(X)$ and $C_G(X)$ the sets of invariant functions in $\bounded(X)$ and $C(X)$, respectively. 
Similarly, we write $C^+(X)$ and $\bounded^+(X)$ for the respective sets of positive functions and abbreviate 
$C_G^+(X):=C^+(X)\cap C_G(X)$ and $\bounded_G^+(X):=\bounded^+(X)\cap \bounded_G(X)$.
For $K\in \hyperhaar(G)$, $f\in C(X)$, and $x\in X$ we denote 
    \[(K^*f)(x):=\frac{1}{\haar{K}}\int_K f(g.x) \dd g\]
and observe that this establishes a linear contraction $K^*\colon C(X)\to \bounded(X)$. 
For a left (or right) F\o lner sequence $\mathcal{F}=(F_n)$ in $G$ we denote 
$\average_\mathcal{F}f:=\limsup_{n\to \infty} F_n^*f$. 
Furthermore, we denote $\average f:=\sup_{\mathcal{F}} \average_\mathcal{F}f$, where the supremum is taken over all left F\o lner sequences. 
We observe that $\average$ and $\average_\mathcal{F}$ are mappings $C(X)\to \bounded(X)$. We next summarize some of their properties, omitting the straightforward proofs.

\begin{proposition}
\label{pro:averageBasicProperties}
    The mapping $\average\colon C(X)\to \bounded(X)$ is 
    \begin{itemize} 
        \item a \emph{contraction}, i.e.\ $\|\average f- \average h\|_\infty\leq \|f-h\|_\infty$ for all $f,h\in C(X)$. 
        \item \emph{monotone}, i.e.\ $\average f\leq \average h$ for all $f,h\in C(X)$ with $f\leq h$. 
        \item \emph{subadditive}, i.e.\ $\average(f+h)\leq \average f + \average h$ for all $f,h\in C(X)$. 
        \item \emph{positively homogeneous}, i.e.\ $\average(\lambda f)=\lambda \average f$ for $f\in C(X)$ and $\lambda\in \mathbb{R}^+$.
        \item \emph{$C_G^+(X)$-multiplicative}, i.e.\ $\average(hf)=h \average f$ for $f\in C(X)$, $h\in C_G^+(X)$. 
        \item \emph{translation invariant}, i.e.\ $\average(f+\lambda)=(\average f)+\lambda$ for $f\in C(X)$ and $\lambda\in \mathbb{R}$. 
        \item \emph{$C_G(X)$-additive}, i.e.\ $\average(h+f)=h+\average f$ for $f\in C(X)$ and $h\in C_G(X)$.   
    \end{itemize}
    The same properties are satisfied for $\average_\mathcal{F}$ for any left (or right) F\o lner sequence $\mathcal{F}$. 
    Furthermore, the mapping $\average\colon C(X)\to \bounded(X)$  
    \begin{itemize}   
        \item is \emph{invariant}, i.e.\ $\average(g^*f)=\average f$ for all $f\in C(X)$ and $g\in G$.\\
        The same property is satisfied for $\average_\mathcal{F}$ for any left F\o lner sequence $\mathcal{F}$.  
        \item satisfies $\average f\in \bounded_G(X)$ for all $f\in C(X)$.\\
        The same property is satisfied for $\average_\mathcal{F}$ for any right F\o lner sequence $\mathcal{F}$. 
    \end{itemize}
\end{proposition}

We next summarize some equivalent approaches for computing $\average$. 

\begin{proposition}
\label{pro:averagesAdifferentApproaches}
    For $f\in C(X)$ we have 
\begin{align*}
\average f
=
\sup_{\mathcal{F}}\limsup_{n\to \infty}F_n^* f
=
\sup_{\mathcal{F}}\liminf_{n\to \infty}F_n^* f
=
\inf_{K\in \hyperhaar(G)}\sup_{g\in G} (Kg)^* f
\end{align*}
where the suprema, taken over all left F\o lner sequences $\mathcal{F}=(F_n)$, are attained pointwise. 
Furthermore, for every left F\o lner sequence $(F_n)$ in $G$ we have that $\sup_{g\in G} (F_ng)^*f$ converges pointwise to $\average f$. 
\end{proposition}
\begin{proof}
    For $f\in C(X)$ there exists $\lambda\in \mathbb{R}$ such that $f+\lambda\geq 0$. 
    Since all terms above are easily seen to be translation invariant, we assume w.l.o.g.\ that $f$ is positive. 
    Fix $x\in X$ and define
    $f_x\colon G\to \mathbb{R}^+$ by $f_x(g):=f(g.x)$. 
    Since $f_x$ is bounded and measurable, Example~\ref{exa:SpropertyForGfunctions} implies that
    $H\colon \hyperhaar(G)\to \mathbb{R}^+$ defined by $H(K):=\int_K f_x(g)\, \dd g$
    is right dominated, right measurable and satisfies~\ref{S}. 
    We have $H(K)/\haar{K}=(K^*f)(x)$  for all $K\in \hyperhaar(G)$ and conclude the statement from Proposition~\ref{pro:subadditivityGeneral}. 
\end{proof}

\begin{proposition}
\label{pro:averagesAandAFPickFolner}
	For $f\in C(X)$ and a countable family $N\subseteq X$ there exists a left F\o lner sequence $\mathcal{F}$ in $G$ with 
	$\average f(x)=\average_\mathcal{F} f(x)$ for all $x\in N$. 
\end{proposition}
\begin{proof}
	Denote $I:=N\times \mathbb{N}$ and note that $I$ is countable. 
	For $i=(x,k)\in I$ we consider a left F\o lner sequence 
    $\mathcal{F}^{(i)}=(F_n^{(i)})$ that satisfies 
	\[\average f(x)\leq \lim_{n\to \infty} (F_n^{(i)})^* f(x)+\tfrac{1}{k}.\]
	Since $(\mathcal{F}^{(i)})_{i\in I}$ is a countable family of left F\o lner sequences, there exists a left F\o lner sequence $\mathcal{F}=(F_n)$ that has a common subsequence with each $\mathcal{F}^{(i)}$ \cite[Proposition 2.6]{fuhrmann2025continuity}. 
	For $x\in N$ and $k\in \mathbb{N}$, we have $i:=(x,k)\in I$ and observe 
    \begin{align*}
    \average f(x)
    &\leq \lim_{n\to \infty} (F_n^{(i)})^*f(x)+\tfrac{1}{k}
    \leq \limsup_{n\to \infty}F_n^* f(x)+\tfrac{1}{k}
    = \average_\mathcal{F} f(x)+\tfrac{1}{k}. 
    \end{align*}
    We thus have $\average f(x)=\average_\mathcal{F} f(x)$ for all $x\in N$. 
\end{proof}

For $f,h\in \bounded^+(X)$ and $\epsilon,\delta>0$, we say that $f$ is \emph{$\epsilon$-$\delta$-majorizing for $h$} if for all $x\in X$ with $f(x)\leq \delta$ we have $h(x)\leq \epsilon$. 

\begin{lemma}\label{lem:averagesMajorizingTechniques}
	Consider $\epsilon,\delta'>0$ and $f,h\in C^+(X)$. \\    
	Whenever $f$ is $\frac{\epsilon}{2}$-$\delta'$-majorizing for $h$, 
    then there exists $\delta>0$ such that 
\begin{itemize} 
\item[(i)] $\average f$ is $\epsilon$-$\delta$-majorizing for $\average h$, and
\item[(ii)] $\average_\mathcal{F} f$ is $\epsilon$-$\delta$-majorizing for $\average_\mathcal{F} h$ for all left (or right) F\o lner sequences $\mathcal{F}$.
\end{itemize}	
\end{lemma}
\begin{proof}
    By the positive homogeneity of $\average_\mathcal{F}$ and $\average$ we assume w.l.o.g.\ that $h\leq 1$. 
    Observe that (i) is a straightforward consequence of (ii). 
    Denote $\delta:=\epsilon\delta'/4$. 
    For $x\in X$ we define positive functions $f_x,h_x\colon G\to \mathbb{R}$ by $f_x(g):=f(g.x)$ and $h_x(g):=h(g.x)$.

    For a left (or right) F\o lner sequence $\mathcal{F}=(F_n)$, we consider $x\in X$ with $\average_\mathcal{F}(f)(x)\leq \delta$. 
    For all sufficiently large $n$ we observe that $F_n^*f(x)\leq 2\delta$.
    Note that $\mu:=\haar{(\cdot)\cap F_n}/\haar{F_n}$ defines a probability measure on $G$.
    We have 
    $2\delta\geq F_n^*f(x)=\mu(f_x)\geq \int_{f_x>\delta'}\delta'd\mu$ and hence 
    $\int_{f_x>\delta'}1 d\mu\leq\epsilon/2$.
    Since $h_x\leq 1$ holds and since $f_x(g)\leq\delta'$ implies $h_x(g)\leq\epsilon/2$ for all $g\in G$ we have
    \begin{align*}
        F_n^*h(x)
        =	\mu(h_x) 
        \leq \int_{f_x\leq \delta'} \frac{\epsilon}{2} d\mu + \int_{f_x> \delta'} 1 d\mu  
        \leq \epsilon. 
    \end{align*}
    We observe $\average_\mathcal{F}(h)(x)=\limsup_{n\to \infty} F_n^*h(x)\leq \epsilon$. 
\end{proof}

The idea of the following lemma can be found in \cite[Remark 2.1]{garcia2019when}. We include the short proof for the convenience of the reader. 

\begin{lemma}
    \label{lem:averagesDensityEstimatesGeneral}
    Let $f\in C^+(X)$ and $\mathcal{F}$ be a left (or right) F\o lner sequence in $G$. Let $x\in X$ and $\epsilon>0$ and denote
    $G_{x,\epsilon}:= \{g\in G;\, f(g.x)>\epsilon\}.$
\begin{itemize}
    \item[(i)] Whenever $\average_\mathcal{F}f(x)\leq \epsilon^2$, then we have
    $\upperAdens_\mathcal{F}(G_{x,\epsilon})\leq \epsilon.$
    \item[(ii)] Whenever
    $\upperAdens_\mathcal{F}(G_{x,\epsilon})\leq \epsilon$, then
    we have 
    $\average_\mathcal{F}f(x)\leq (\|f\|_\infty+1)\epsilon$. 
    \item[(iii)] Whenever $\average f(x)\leq \epsilon^2$, then we have
    $\upperBdens(G_{x,\epsilon})\leq \epsilon.$
    \item[(iv)] Whenever
    $\upperBdens(G_{x,\epsilon})\leq \epsilon$, then
    we have 
    $\average f(x)\leq (\|f\|_\infty+1)\epsilon$. 
\end{itemize}
\end{lemma}
\begin{proof}
    Observe that $G_{x,\epsilon}$ is open, hence measurable. \\
    (i): We compute
    \begin{align*}
        \upperAdens_\mathcal{F}(G_{x,\epsilon})
        \leq \frac{1}{\epsilon} \limsup_{n\to \infty}\frac{1}{\haar{F_n}}\int_{F_n\cap G_{x,\epsilon}} f(g.x) \dd g
        \leq \frac{1}{\epsilon}\average_\mathcal{F}f(x)
        \leq \epsilon
    \end{align*}
    (ii): We observe 
    \begin{align*}
        \average_\mathcal{F}f(x)
        &\leq  
            \limsup_{n\to \infty}\frac{1}{\haar{F_n}}\int_{F_n\cap G_{x,\epsilon}} f(g.x) \dd g
            + \limsup_{n\to \infty}\frac{1}{\haar{F_n}}\int_{F_n\setminus G_{x,\epsilon}} f(g.x) \dd g\\
        &\leq \upperAdens_\mathcal{F}(G_{x,\epsilon})\|f\|_\infty+\epsilon\upperAdens_\mathcal{F}(G\setminus G_{x,\epsilon})
        \leq (\|f\|_\infty+1)\epsilon. 
\end{align*}
The statements of (iii) and (iv) are direct consequences of (i) and (ii).
\end{proof}

\section{The mean diameter}
\label{sec:meanDiameter}

Let $G$ be an amenable \lcsigma group. 
Consider an action $X$ of $G$ on a compact metric space. 
Let $d$ be a continuous metric on $X$, and note that $\diam_d \in C^+(\hyper(X))$. 
Considering the induced action of $G$ on $\hyper(X)$, the \emph{mean diameter (of $d$)} is given by 
$\Diam_d:=\average(\diam_d)$. 
Similarly, for a left (or right) F\o lner sequence $\mathcal{F}$, we consider the \emph{mean diameter (of $d$) w.r.t.\ $\mathcal{F}$}, which is defined by  
$\Diam_{\mathcal{F},d}:=\average_\mathcal{F}(\diam_d)$. 
If the metric is clear from the context, we simply write $\Diam$ and $\Diam_\mathcal{F}$.

\begin{remark}
\label{rem:invarianceOfDiam}
    Proposition \ref{pro:averageBasicProperties} implies that for any $A\in \hyper(X)$ and $g\in G$ we have $\Diam(g.A)=\Diam(A)$, as well as $\Diam_\mathcal{F}(g.A)=\Diam_\mathcal{F}(A)$ for any right F\o lner sequence $\mathcal{F}$.   
\end{remark}

The following result is a direct consequence of Proposition \ref{pro:averagesAdifferentApproaches}.

\begin{proposition}
\label{pro:descriptionDiam}
    For $A\in \hyper(X)$ we have 
    \begin{align*}
    \Diam(A)
    &= \sup_\mathcal{F}\Diam_\mathcal{F}(A)
    = \sup_\mathcal{F}\liminf_{n\to \infty}\frac{1}{\haar{F_n}}\int_{F_n}\diam(g.A) \dd g\\
    &=
    \inf_{K\in \hyperhaar(G)}\sup_{h\in G}\frac{1}{\haar{Kh}}\int_{Kh}\diam(g.A)\dd g,
    \end{align*}
    where the suprema, taken over all left F\o lner sequences $\mathcal{F}=(F_n)$, are attained.
    Furthermore, for every left F\o lner sequence $(F_n)$, we have  
    \[\sup_{h\in G}\frac{1}{\haar{F_nh}}\int_{F_nh}\diam(g.A)\dd g \to \Diam(A).\]
\end{proposition}
\begin{remark}
    The F\o lner sequences attaining the suprema in Proposition \ref{pro:descriptionDiam} depend on $A$.
\end{remark}

\section{Regularity and diam-mean proximality}
\label{sec:regularityVSmeanDiamProximality}

For a factor map $\pi\colon X\to Y$, we call $x\in X$ an \emph{injectivity point} of $\pi$ whenever $\pi^{-1}(\pi(x))=\{x\}$. 
Note that the set of injectivity points is a $G_\delta$ set\footnote{A subset of a topological space is called \emph{$G_\delta$} if it is the countable intersection of open sets.} \cite[Lemma 1.1]{akin2001residual} and hence measurable. 
A factor map $\pi\colon X\to Y$ is called \emph{regular} if for all $\mu\in \mathcal{M}_G(X)$ the set of injectivity points $X_0$ satisfies $\mu(X_0)=1$.
From the ergodic representation we observe that 
$\pi$ is regular
if and only if 
$\mu(X_0)=1$ for all ergodic $\mu\in \mathcal{M}_G(X)$. See \cite[Lemma 6]{farrell1962representation} for details.

\begin{remark}
\label{rem:regularityViaY}
    Denote $Y_0:=\pi(X_0)$ and note that $\pi(X\setminus X_0)=Y\setminus \pi(X_0)$.
    Since closed maps map $F_\sigma$ sets\footnote{
    A subset of a topological space is called \emph{$F_\sigma$} if it is the
    complement of a $G_\delta$ set, i.e.\ the countable union of closed sets.
    } to $F_\sigma$ sets, it follows that $Y_0$ is $G_\delta$ and hence measurable. 

    Since we assume $G$ to be amenable, for ergodic $\nu\in \mathcal{M}_G(Y)$ there exists $\mu\in \mathcal{M}_G(X)$ with $\pi_*\mu=\nu$. 
    We thus observe that 
    $\pi$ is regular
    if and only if 
    $\nu(Y_0)=1$ for all ergodic $\nu\in \mathcal{M}_G(Y)$. 
    Another application of the ergodic representation yields that 
    $\pi$ is regular
    if and only if 
    $\nu(Y_0)=1$ for all $\nu\in \mathcal{M}_G(Y)$.
\end{remark}

Let $\pi\colon X\to Y$ be a factor map. 
We denote $\hyper_\pi(X)$ for the set of all $A\in \hyper(X)$ for which there exists $y\in Y$ with $A\subseteq \pi^{-1}(y)$. Observe that $\hyper_\pi(X)$ and the set of singletons $\single(X)$ are closed invariant (non-empty) subsets of $\hyper(X)$ and that $\single(X)\subseteq \hyper_\pi(X)$.
In the following we consider the induced action on $\hyper_\pi(X)$ and $\single(X)$. Note that $X$ is conjugated to $\single(X)$ via $x\mapsto \{x\}$. 

For $A\in \hyper_\pi(X)$ there exists a unique $y_A\in Y$ with $\pi(A)=\{y_A\}$. 
The mapping $\hyper_\pi(X)\to Y$ that assigns $A\mapsto y_A$ will be denoted by $\pi_\hyper$. It is straightforward to verify that $\pi_\hyper\colon \hyper_\pi(X)\to Y$ is a factor map. 

\begin{theorem}
    \label{the:regularFactorMapCharacterization}
    For a factor map $\pi\colon X\to Y$ the following statements are equivalent. 
    \begin{itemize}
        \item[(i)] $\pi$ is regular.
        \item[(ii)] $\pi$ is diam-mean proximal, i.e.\ for all $y\in Y$ we have $\Diam(\pi^{-1}(y))=0$. 
        \item[(ii')] For all $y\in Y$ and all $\epsilon>0$ we have 
        \[\upperBdens\left(\left\{g\in G;\, \diam(\pi^{-1}(g.y))>\epsilon\right\}\right)\leq \epsilon.\]
        \item[(iii)] There exists a left F\o lner sequence $\mathcal{F}$ such that $\pi$ is $\mathcal{F}$-diam-mean proximal, i.e.\ such that for all $y\in Y$ we have $\Diam_\mathcal{F}(\pi^{-1}(y))=0$.
        \item[(iii')] There exists a left F\o lner sequence $\mathcal{F}$ such that for all $y\in Y$ and all $\epsilon>0$ we have 
        \[\upperAdens_\mathcal{F}\left(\left\{g\in G;\, \diam(\pi^{-1}(g.y))>\epsilon\right\}\right)\leq \epsilon.\]
        \item[(iv)] $\pi_\hyper\colon \hyper_\pi(X)\to Y$ is regular. 
        \item[(v)] $\mathcal{M}_G(\hyper_\pi(X))=\mathcal{M}_G(\single(X))$. 
    \end{itemize}
\end{theorem}
\begin{remark}
\label{rem:densitySetsMeasurability}
    Since $G$ induces an action on $\hyper(X)$ and $\diam\in C^+(\hyper(X))$ we observe that for $A\in \hyper(X)$ the set
    $\{g\in G;\, \diam(g.A)>\epsilon\}$ 
    is open in $G$ and hence measurable. From $g.\pi^{-1}(y)=\pi^{-1}(g.y)$ it follows that the subsets of $G$ considered in (ii') and (iii') are measurable. 
\end{remark}
\begin{proof}[Proof of Theorem \ref{the:regularFactorMapCharacterization}:]     
    Denote $X_0$ for the set of injectivity points of $\pi$ and $Y_0:=\pi(X_0)$. 
    It is straightforward to verify that the injectivity points of $\pi_\hyper$ are given by
    $\single(X_0)=\pi_\hyper^{-1}(Y_0)$.
    With this observation it follows from Remark \ref{rem:regularityViaY} that (i) and (iv) are equivalent. 
    
(iv)$\Rightarrow$(v): 
    For $\mu\in \mathcal{M}_G(\hyper_\pi(X))$ we have that $\mu(\single(X_0))=1$ and hence $\mu(\single(X))=1$. This shows $\mu\in \mathcal{M}_G(\single(X))$. 
    
(v)$\Rightarrow$(ii): 
    Let $A\in \hyper_\pi(X)$ and consider a left F\o lner sequence $\mathcal{F}=(F_n)$.
    Choose a subsequence $\mathcal{F}'=(F_n')$ such that 
    $\Diam_\mathcal{F}(A)=\lim_{n\to \infty} (F_n')^*\diam (A)$. 
    By a standard Krylov-Bogolyubov argument and by possibly considering a further subsequence we assume w.l.o.g.\ that 
    $A$ is $\mathcal{F}'$-generic for some $\mu\in \mathcal{M}_G(\hyper_\pi(X))=\mathcal{M}_G(\single(X))$.
    Since $\diam\equiv 0$ on $\single(X)$ we thus have 
    \begin{align*}
        0
        =\mu(\diam)
        =\lim_{n\to \infty} ((F_n')_*\delta_A)(\diam)
        =\lim_{n\to \infty} (F_n')^*\diam(A)
        =\Diam_\mathcal{F}(A).
    \end{align*}
    This shows $\Diam_{\mathcal{F}}(A)=0$ for any left F\o lner sequence $\mathcal{F}$. 
    Taking the supremum over all left F\o lner sequences we observe 
    $\Diam(A)=\sup_{\mathcal{F}}\Diam_\mathcal{F}(A)=0$. 

(ii)$\Rightarrow$(iii): 
    Follows from $\Diam_\mathcal{F}\leq \Diam$ for any left F\o lner sequence $\mathcal{F}$. 

(iii)$\Rightarrow$(i): 
    Any left F\o lner sequence $\mathcal{F}$ allows for a tempered\footnote{
        A left F\o lner sequence $\mathcal{F}=(F_n)$ is called \emph{tempered} if
        $\sup_{n\in \mathbb{N}} {\haar{\bigcup_{k<n} F_k^{-1}F_n}}/{\haar{F_n}}<\infty.$
        We will not need this definition and only use that the pointwise ergodic theorem holds along $\mathcal{F}$. 
        See \cite{lindenstrauss2001pointwise} for more details. 
    } 
    subsequence $\mathcal{F}'$ \cite[Proposition 1.4]{lindenstrauss2001pointwise}.
    Since $\Diam_{\mathcal{F}'}\leq \Diam_\mathcal{F}$ holds we assume w.l.o.g.\ that $\mathcal{F}$ is tempered. 

    If $\pi$ is not regular, then there exist an ergodic $\nu\in \mathcal{M}_G(Y)$ with $\nu(Y_0)<1$. Since $Y_0$ is invariant, we observe $\nu(Y_0)=0$, i.e.\ $\nu(Y\setminus Y_0)=1$. 
    Denote $A_\epsilon:=\{y\in Y;\, \diam(\pi^{-1}(y))\geq \epsilon\}$ for $\epsilon>0$. 
    From $Y\setminus Y_0= \bigcap_{n\in \mathbb{N}}A_{1/n}$ we observe that there exists $\epsilon>0$ with $\nu(A_\epsilon)>0$. 

    By the Lindenstrauss ergodic theorem \cite[Theorem 1.2]{lindenstrauss2001pointwise} there exists $y\in Y$ such that 
    ${\haar{F_n}}^{-1}\int_{F_n} \mathbf{1}_{A_\epsilon}(g.y)\dd g \to \nu(A_\epsilon)$. 
    Denoting 
    $F_n':=\{g\in F_n;\, g.y\in A_\epsilon\}$
    we observe
    \begin{align*}
        \Diam_\mathcal{F}(\pi^{-1}(y))
        &= \limsup_{n\to \infty} \frac{1}{\haar{F_n}}\int_{F_n} \diam(g.\pi^{-1}(y))\dd g\\
        &\geq \limsup_{n\to \infty} \frac{1}{\haar{F_n}}\int_{F_n'} \diam(\pi^{-1}(g.y))\dd g\\
        &\geq \lim_{n\to \infty} \frac{1}{\haar{F_n}}\int_{F_n} \epsilon \mathbf{1}_{A_\epsilon}(g.y)\dd g
        =\epsilon\nu(A_\epsilon)>0.
    \end{align*}
    
(ii)$\Rightarrow$(ii'): 
    For $y\in Y$ and $\epsilon>0$ consider $f:=\diam\in C(\hyper(X))$ and $A:=\pi^{-1}(y) \in \hyper(X)$. We have 
    $\average f(A)=\Diam(\pi^{-1}(y))=0\leq \epsilon^2$ and Lemma~\ref{lem:averagesDensityEstimatesGeneral}(iii) yields 
    \[
    \upperBdens(\{g\in G;\, \diam(\pi^{-1}(g.x))>\epsilon\})
    =\upperBdens(\{g\in G;\, f(g.A)>\epsilon\})
    \leq \epsilon.\]

(ii')$\Rightarrow$(ii): 
    Consider $y\in Y$ and $\epsilon>0$. 
    With $f:=\diam$ and $A:=\pi^{-1}(y)$ we have 
    $
    \upperBdens(\{g\in G;\, f(g.A)>\epsilon\})
    \leq \epsilon
    $
    and Lemma~\ref{lem:averagesDensityEstimatesGeneral}(iv) yields 
    \[
    \Diam(\pi^{-1}(y))=\average f(A)\leq (\|f\|_\infty+1)\epsilon=(\diam(X)+1)\epsilon. 
    \]
    Since $\epsilon>0$ is arbitrary we observe $\Diam(\pi^{-1}(y))=0$.

(iii)$\Leftrightarrow$(iii'): 
    Follows with a similar argument from Lemma \ref{lem:averagesDensityEstimatesGeneral}. 
\end{proof}

\begin{remark}
    Theorem \ref{the:regularFactorMapCharacterization} also shows that the notion of diam-mean proximality of a factor map $\pi\colon X\to Y$ is independent of the choice of a metric on $X$. A similar statement holds about $\mathcal{F}$-diam-mean proximality for any left F\o lner sequence $\mathcal{F}$. 
    Note that this can also be seen directly from Lemma \ref{lem:averagesMajorizingTechniques}.
\end{remark}

\begin{remark}
    A closely related (but not equivalent) property of factor maps is \emph{Banach proximality}. It was introduced in \cite{li2015mean} and shown to be equivalent to topo-isomorphy in \cite{qiu2020note, fuhrmann2022structure, hauser2024mean}. 
    See Subsection \ref{subsec:meanEquicontinuity} for more details on these notions. 
\end{remark}

\section{Diam-mean equicontinuity}
\label{sec:diamMeanEquicontinuity}

Let $X$ be an action (of $G$) and $d$ be a continuous metric on $X$. 
We say that $x\in X$ is \emph{diam-mean equicontinuous (w.r.t.\ $d$)} if for each $\epsilon>0$ there exists a closed neighbourhood $A$ of $x$ such that 
$\Diam_d(A)<\epsilon$. The following proposition shows that diam-mean equicontinuity is independent of the choice of $d$. We will thus simply speak of diam-mean equicontinuous points. 

\begin{proposition}
\label{pro:FdiamMeanEquicontinuousPointGeneratorIndependence}
    Let $X$ be an action and $d$ be a continuous metric on $X$. 
    A point $x\in X$ is diam-mean equicontinuous w.r.t.\ $d$ if and only if it is diam-mean equicontinuous w.r.t.\ any continuous metric on $X$. 
\end{proposition}
\begin{proof}
    Assume that $x$ is diam-mean equicontinuous w.r.t.\ $d$ and let $d'$ be another continuous metric on $X$. 
    For $\epsilon>0$ there exists $\delta'>0$ such that $d$ is  $\frac{\epsilon}{2}$-$\delta'$-majorizing for $d'$. 
    We thus observe that $\diam_{d}$ is $\frac{\epsilon}{2}$-$\delta'$-majorizing for $\diam_{d'}$. 
    From Lemma \ref{lem:averagesMajorizingTechniques} it follows that there exists $\delta>0$ such that $\Diam_{d}$ is $\epsilon$-$\delta$-majorizing for $\Diam_{d'}$. 
    Since $x$ is a diam-mean equicontinuous point w.r.t.\ $d$ there exists a compact neighbourhood $A$ of $x$ such that $\Diam_d(A)\leq \delta$.
    For such $A$ we observe $\Diam_{d'}(A)\leq \epsilon$.    
\end{proof}

An action $X$ is called \emph{diam-mean equicontinuous} if all points $x\in X$ are diam-mean equicontinuous. 
Similarly we define the notion of $\mathcal{F}$-diam-mean equicontinuity (for points and actions) for any left (or right) F\o lner sequence $\mathcal{F}$, by replacing $\Diam_d$ with $\Diam_{\mathcal{F},d}$.
Note that a similar argument shows the independence from the choice of a metric also in this context and that any diam-mean equicontinuous point is $\mathcal{F}$-diam-mean equicontinuous.

\begin{remark}
\label{rem:diamMeanEquicontinuityGlobalization}
    It is straightforward to show that an action $X$ is diam-mean equicontinuous if and only if for all $\epsilon>0$ there exists $\delta>0$ such that for all $x\in X$ we have $\Diam(\ball_\delta(x))\leq \epsilon$. A similar statement holds for $\mathcal{F}$-diam-mean equicontinuity for any left (or right) F\o lner sequence $\mathcal{F}$. 
\end{remark}

\begin{remark}
    Any equicontinuous action is diam-mean equicontinuous.
\end{remark}

\begin{remark}
    \label{rem:transitivityDiamMEanEquicontinuous}
    From the invariance of $\Diam$ (Remark \ref{rem:invarianceOfDiam}) it is straightforward to observe that any $x\in X$ for which there exists a diam-mean equicontinuous point (w.r.t.\ $X$) in $\overline{O(x)}$ is diam-mean equicontinuous (w.r.t.\ $X$). 
    In particular, we observe the following. 
    \begin{itemize}
        \item[(i)] The set of diam-mean equicontinuous points is invariant. 
        \item[(ii)] Whenever $X$ contains a diam-mean equicontinuous point, then any transitive point in $X$ is diam-mean equicontinuous. 
        \item[(iii)] A minimal action is diam-mean equicontinuous if and only if it allows for a diam-mean equicontinuous point. 
    \end{itemize}
    A similar statement holds for $\mathcal{F}$-diam-mean equicontinuity whenever $\mathcal{F}$ is a right F\o lner sequence. 
\end{remark}

\begin{remark}
\label{rem:FdiamMeanEquiVSDiamMeanEqui}
    It was shown in \cite{garcia2021mean,  haupt2025multivariate} that a minimal action of an Abelian group is diam-mean equicontinuous if and only if it allows for a $\mathcal{F}$-diam-mean equicontinuous point for some F\o lner sequence $\mathcal{F}$. For details see \cite[Theorem 4.12]{garcia2021mean} and \cite[Theorem 5.6]{haupt2025multivariate} (with $m=2$).  
\end{remark}

\begin{remark}
    It remains open whether the $\mathcal{F}$-diam-mean equicontinuity for some F\o lner sequence implies diam-mean equicontinuity in the context of actions of amenable \lcsigma groups. See \cite{haupt2025multivariate, haupt2025note} for some advances in this direction. 
\end{remark}

\begin{remark}
    Diam-mean equicontinuity is also called \emph{Banach diam-mean equicontinuity} \cite{garcia2021mean}. 
    Furthermore, our terminology differs from that of \cite{garcia2021mean}, where minimal actions of $\mathbb{Z}$ (actually actions of the semigroup $\mathbb{N}$) were studied. 
    For the standard F\o lner sequence $\mathcal{F}=(F_n)$ with $F_n=\{0,\dots, n-1\}$ our concept of $\mathcal{F}$-diam-mean equicontinuity is simply called ``diam-mean equicontinuity'' in \cite{garcia2021mean}.   
\end{remark}

In the context of actions of $\mathbb{Z}$ the following characterization was presented in \cite[Lemma 4.4]{garcia2021mean} and \cite[Lemma 4.7]{garcia2021mean}. 
It is a straightforward consequence of Lemma \ref{lem:averagesDensityEstimatesGeneral} and we will need it for the proof of Theorem \ref{the:INTROcharacterizationDMEvsRegular}. 

\begin{proposition}
    \label{pro:characterizationDiamMeanEquicontinuityViaDensities}
    Let $X$ be an action. 
    \begin{itemize}
        \item[(i)] A point $x\in X$ is diam-mean equicontinuous if and only if for any $\epsilon>0$ there exists a closed neighbourhood $A$ of $x$ such that 
        \[\upperBdens(\{g\in G;\, \diam(g.A)>\epsilon\})\leq \epsilon.\]
        \item[(ii)] Let $\mathcal{F}$ be a left (or right) F\o lner sequence in $G$.  
        A point $x\in X$ is $\mathcal{F}$-diam-mean equicontinuous if and only if for any $\epsilon>0$ there exists a closed neighbourhood $A$ of $x$ such that 
        \[\upperAdens_\mathcal{F}(\{g\in G;\, \diam(g.A)>\epsilon\})\leq \epsilon.\]
    \end{itemize}
\end{proposition}
 
In order to show Theorem \ref{the:INTROcharacterizationDMEvsRegular} we will also need the following characterization. 

\begin{proposition}
\label{pro:characterizationDiamMeanEquicontinuityViaAlFolner}
    Let $X$ be an action. 
    \begin{itemize}
        \item[(i)] A point $x\in X$ is diam-mean equicontinuous if and only if $x$ is $\mathcal{F}$-diam-mean equicontinuous w.r.t.\ all left F\o lner sequences $\mathcal{F}$. 
        \item[(ii)]  $X$ is diam-mean equicontinuous if and only if $X$ is $\mathcal{F}$-diam-mean equicontinuous w.r.t.\ all left F\o lner sequences $\mathcal{F}$. 
    \end{itemize}
\end{proposition}
\begin{proof}
    Observe that (ii) is a direct consequence of (i).     
    To show (i) let $\epsilon>0$ and note that $\{\ball_{1/k}(x);\, k\in \mathbb{N}\}$ is a countable family in $\hyper(X)$. Thus, Proposition \ref{pro:averagesAandAFPickFolner} yields the existence of a left F\o lner sequence $\mathcal{F}$ with 
    \begin{align*}
    \Diam\left(\ball_{\frac{1}{k}}(x)\right)
    =\average(\diam)\left(\ball_{\frac{1}{k}}(x)\right)
    =\average_\mathcal{F}(\diam)\left(\ball_{\frac{1}{k}}(x)\right)
    =\Diam_{\mathcal{F}}\left(\ball_{\frac{1}{k}}(x)\right)
    \end{align*}
    for all $k\in \mathbb{N}$. 
    Since $x$ is assumed to be $\mathcal{F}$-diam-mean equicontinuous, there exists $k\in \mathbb{N}$ such that 
    $\Diam(\ball_{1/k}(x))=\Diam_\mathcal{F}(\ball_{1/k}(x))\leq \epsilon$, showing $x$ to be diam-mean equicontinuous. 
\end{proof}

The following example illustrates that $\mathcal{F}$-diam-mean equicontinuity of points depends on the choice of a left F\o lner sequence. 

\begin{example}
    Consider $G:=\mathbb{Z}$. 
    Let $X=\mathbb{Z}\cup \{-\infty,\infty\}$ be the two-point compactification. We act on $X$ with $G$ by $g.x:=g+x$ for $g\in G$ and $x\in \mathbb{Z}$ and by fixing $+\infty$ and $-\infty$. 
    Consider the F\o lner sequences $\mathcal{F}$ and $\mathcal{F}'$ given by
    $F_n:=\{0,\dots,n\}$ and $F_n':=\{-n,\dots,0\}$. 
    It is straightforward to observe that $\infty$ is $\mathcal{F}$-diam-mean equicontinuous, but not $\mathcal{F}'$-diam-mean equicontinuous. 
    We also observe that $\infty$ is not diam-mean equicontinuous, while it is $\mathcal{F}$-diam-mean equicontinuous.  
\end{example}

\section{Regularity and diam-mean equicontinuity}
\label{sec:regularityVSdiamMeanEquicontinuity}

In this section, we will prove Theorem \ref{the:INTROcharacterizationDMEvsRegular}. To do so, we begin by recalling some details from the theory of mean equicontinuous actions. See \cite{fuhrmann2022structure} and the references therein for more details on mean equicontinuity.

\subsection{Mean equicontinuity}
\label{subsec:meanEquicontinuity}

Let $X$ be an action, and let $d$ be a continuous metric on $X$. 
We consider the induced action of $G$ on $X^2$. 
The \emph{Weyl pseudo-metric} $D$ is defined by $D:=\average(d)$. 
An action is called \emph{mean equicontinuous} whenever for all $x\in X$ and $\epsilon>0$ there exists $\delta>0$ such that for all $x'\in \ball_\delta(x)$ we have $D(x,x')\leq \epsilon$. Similarly, we define the \emph{$\mathcal{F}$-Besicovitch pseudo-metric} $D_\mathcal{F}:=\average_\mathcal{F}(d)$ and the concept of \emph{$\mathcal{F}$-mean equicontinuity} for any left F\o lner sequence.
Clearly, any diam-mean equicontinuous action is mean equicontinuous. 
Furthermore, any $\mathcal{F}$-diam-mean equicontinuous action is $\mathcal{F}$-mean equicontinuous.
With an argument similar to that in the proof of Proposition \ref{pro:descriptionDiam} we observe the following. See \cite[Section 3.3]{lacka2018quasi} for the statement in the context of countable amenable groups. 

\begin{remark}
\label{rem:descriptionD}
    For $(x,x')\in X^2$ we have 
    \begin{align*}
    D(x,x')
    &= 
    \sup_\mathcal{F}D_\mathcal{F}(x,x')
    =
    \sup_\mathcal{F}\liminf_{n\to \infty}\frac{1}{\haar{F_n}}\int_{F_n}d(g.x,g.x') \dd g
    \\
    &=
    \inf_{K\in \hyperhaar(G)}\sup_{h\in G}\frac{1}{\haar{Kh}}\int_{Kh}d(g.x,g.x')\dd g,
    \end{align*}
    where the suprema, taken over all left F\o lner sequences $\mathcal{F}=(F_n)$, are attained.
    Furthermore, for every left F\o lner sequence $(F_n)$, we have 
    \[
        \sup_{h\in G}\frac{1}{\haar{F_nh}}\int_{F_nh}d(g.x,g.x')\dd g \to D(x,x').
    \]
\end{remark}

A factor map $\pi\colon X\to Y$ is called \emph{Banach proximal} if for all $y\in Y$ and $x,x'\in \pi^{-1}(y)$ we have $D(x,x')=0$ \cite{li2015mean, hauser2024mean}. 
It follows from Theorem~\ref{the:regularFactorMapCharacterization} that any regular factor map is Banach proximal. 
\begin{remark}
    A factor map $\pi\colon X\to Y$ is called \emph{topo-isomorphic} if for each $\mu\in \mathcal{M}_G(X)$ the map $\pi$ induces a 
    measure-theoretic isomorphism between $(X,\mu)$ and $(Y,\pi_*\mu)$.
    It was shown\footnote{
    Although \cite{hauser2024mean} presents the argument only for countable amenable groups, it extends straightforwardly to the context of amenable \lcsigma groups. 
    } 
    in \cite[Theorem 8.1]{hauser2024mean} that a factor map $\pi\colon X\to Y$ is Banach proximal if and only if it is topo-isomorphic.
    See \cite{downarowicz2016isomorphic, fuhrmann2022structure, hauser2024mean} for further details on topo-isomorphic factor maps. 
\end{remark}

We will need the following results of \cite{fuhrmann2022structure}. 

\begin{proposition}\cite[Theorem 3.16]{fuhrmann2022structure}
\label{pro:meanEquicontinuousImpliesBanachProximalToMEF}
\begin{itemize}
    \item[(i)] If $\pi\colon X\to Y$ is a Banach proximal factor map onto an equicontinuous action $Y$, then $\pi$ is the factor map onto the maximal equicontinuous factor of $X$. 
    \item[(ii)] An action $X$ is mean equicontinuous if and only if the factor map onto the maximal equicontinuous factor is Banach proximal.
\end{itemize}
\end{proposition}

If $X$ is an action, then there exists a smallest closed subset $A\subseteq X$ such that $\mu(A)=1$ for all $\mu\in \mathcal{M}_G(X)$. We call $A$ the \emph{maximal support of $X$}. 
For reference regarding the following, see \cite[Proposition 5.7]{fuhrmann2022structure} and \cite[Theorem 5.8]{fuhrmann2022structure}. 

\begin{proposition}
    \label{pro:meanEquicontinuityResults}
    Let $X$ be an action and consider a left F\o lner sequence $\mathcal{F}$. 
    \begin{itemize}
        \item[(i)] If $X$ is $\mathcal{F}$-mean equicontinuous, then any point $x$ in the maximal support of $X$ is $\mathcal{F}$-generic.
        \item[(ii)] If $X$ is mean equicontinuous, then any $x\in X$ is $\mathcal{F}$-generic. 
    \end{itemize}
\end{proposition}

\subsection{The maximal equicontinuous factor of a diam-mean equicontinuous action}

In order to relate diam-mean equicontinuity to regularity we need to relate the mean diameter of fibres to that of balls. For Banach proximal factor maps this will be achieved in the following lemmas. 

\begin{lemma}
\label{lem:meanDiameterFibresVsBallsForBPfactorMapPreparationsII}
    Let $X$ be an action. 
    For $A\in \hyper(X)$, $K\in \hyperhaar(G)$, and $\epsilon>0$ there exists a finite non-empty subset 
    $E\subseteq A$ such that 
    \[K^*\diam(A)\leq K^*\diam(E) +\epsilon.\]
\end{lemma}
\begin{proof}
    For any $g\in G$, we choose $x_g,x_g'\in A$ with $d(g.x_g,g.x_g')=\diam(g.A)$. 
    Recall that $\diam\colon \hyper(X)\to \mathbf{R}^+$ is continuous. 
    Furthermore, note that 
    $K\times X\to X$
    and 
    $K\times \hyper(X)\to \hyper(X)$ are uniformly continuous. 
    Thus, there exists a neighbourhood $V$ of the identity element of $G$ such that
    for all $x\in X$, $A'\in \hyper(X)$ and $g,g'\in K$ with $g'g^{-1}\in V$ we have 
    $d(g.x,g'.x)\leq \epsilon/3$ and $|\diam(g.A')-\diam(g'.A')|\leq \epsilon/3$. 
    Choose a finite subset $F$ of $K$ with $K\subseteq VF$ and denote
    $E:=\{x_h,x_h';\, h\in F\}$. 
    For $g\in K$ there exists $h\in F$ with $gh^{-1}\in V$.
    Such $h$ satisfies 
    \[\diam(h.A)=d(h.x_h,h.x_h')=\diam(h.E)\leq \diam(g.E)+\tfrac{1}{3}\epsilon\]     
    and allows to compute
    \begin{align*}
        \diam(g.A)
        &=d(g.x_g,g.x_g')\\
        &\leq d(g.x_g,h.x_g) + d(h.x_g,h.x_g') + d(h.x_g',g.x_g')\\
        &\leq \diam(h.A) + \tfrac{2}{3}\epsilon
        \leq \diam(g.E) + \epsilon. 
    \end{align*}
    This shows that for all $g\in K$ we have 
    $\diam(g.A)\leq \diam(g.E)+\epsilon$. We thus conclude that $K^*\diam(A)\leq K^*\diam(E)+\epsilon$. 
\end{proof}

\begin{lemma}
\label{lem:meanDiameterFibresVsBallsForBPfactorMapPreparationsI}
    Let $\pi\colon X\to Y$ be a Banach proximal factor map and $y\in Y$. 
    For all finite non-empty subsets $E$ of $\pi^{-1}(y)$ we have $\Diam(E)=0$. 
    In particular, for $\delta>0$ there exists $g\in G$ with $\diam(g.E)\leq \delta$. 
\end{lemma}
\begin{proof}
    For $g\in G$ we have 
    \[\diam(g.E)=\max_{x,x'\in E}d(g.x,g.x')\leq \sum_{x,x'\in E}d(g.x,g.x').\] 
    We thus observe $\Diam(E)\leq \sum_{x,x'\in E}D(x,x')=0$ from a straightforward subadditivity argument. 
\end{proof}

\begin{lemma}
\label{lem:meanDiameterFibresVsBallsForBPfactorMapPreparationsIII}
    Let $\pi\colon X\to Y$ be a Banach proximal factor map, $y\in Y$ and $\delta>0$. 
    For $K\in \hyperhaar(G)$ and $\epsilon>0$ there exist $g\in G$ and $x\in X$ with 
        \[K^*\diam(\pi^{-1}(y))\leq (Kg)^*\diam(\ball_\delta(x))+\epsilon.\] 
\end{lemma}
\begin{proof}
    By Lemma \ref{lem:meanDiameterFibresVsBallsForBPfactorMapPreparationsII} there exists a finite non-empty subset $E\subseteq \pi^{-1}(y)$ with $K^*\diam(\pi^{-1}(y))\leq K^*\diam(E)+\epsilon$. 
    From Lemma \ref{lem:meanDiameterFibresVsBallsForBPfactorMapPreparationsI} we know that there exists $g\in G$ with 
    $\diam(g^{-1}.E)\leq \delta$. 
    For any $x\in g^{-1}.E$ we observe $g^{-1}.E\subseteq \ball_\delta(x)$, i.e.\ $E\subseteq g.\ball_\delta(x)$ and compute
    $
    K^*\diam(\pi^{-1}(y))-\epsilon
        \leq  K^*\diam(E)
        \leq (K^*\diam)(g.\ball_\delta(x))
        = ((Kg)^*\diam)(\ball_\delta(x)).
    $
\end{proof}

\begin{lemma}
\label{lem:meanDiameterFibresVsBallsForBPfactorMap}
    Let $\pi\colon X\to Y$ be a Banach proximal factor map, $y\in Y$ and $\delta>0$. 
    We have 
    $\Diam(\pi^{-1}(y))\leq \sup_{x\in X}\Diam(\ball_\delta(x))$.
\end{lemma}
\begin{proof}
    Let $\hat{\mathcal{F}}=(\hat{F}_n)$ be a left F\o lner sequence.
    Consider a subsequence $\mathcal{F}=(F_n)$ such that $\Diam_{\hat{\mathcal{F}}}(\pi^{-1}(y))=\lim_n F_n^*\diam(\pi^{-1}(y))$ and a sequence $(\epsilon_n)$ in $(0,1)$ satisfying $\epsilon_n\to 0$. 
    From Lemma \ref{lem:meanDiameterFibresVsBallsForBPfactorMapPreparationsIII} we observe the existence of sequences $(g_n)$ in $G$ and $(x_n)$ in $X$ such that 
    \[F_n^*\diam(\pi^{-1}(y))\leq (F_ng_n)^*\diam(\ball_{\delta/2}(x_n))+\epsilon_n.\]
    Since $X$ is compact (and by possibly considering a further subsequence of $\mathcal{F}$) we assume w.l.o.g.\ that $(x_n)$ converges to some $x\in X$. 
    For large $n$ we have $\ball_{\delta/2}(x_n)\subseteq \ball_\delta(x)$ and hence
    \begin{align*}
        F_n^*\diam(\pi^{-1}(y))-\epsilon_n
        \leq (F_ng_n)^*\diam(\ball_{\delta/2}(x_n))
        \leq (F_ng_n)^*\diam(\ball_{\delta}(x)). 
    \end{align*}
    Since $\mathcal{F}':=(F_n g_n)$ is a left F\o lner sequence it follows that
    \begin{align*}
        \Diam_{\hat{\mathcal{F}}}(\pi^{-1}(y))
        =\Diam_\mathcal{F}(\pi^{-1}(y))
        \leq \Diam_{\mathcal{F}'}(\ball_\delta(x))
        \leq \sup_{x\in X}\Diam(\ball_\delta(x)).
    \end{align*}
    Taking the supremum over all left F\o lner sequences $\hat{\mathcal{F}}$ yields the statement. 
\end{proof}

\begin{proposition}
\label{pro:diamMeanEquicontinuousActionhasRegularMEF}
    Let $X$ be a diam-mean equicontinuous action. Then the factor map $\pi$ onto the maximal equicontinuous factor of $X$ is regular. 
\end{proposition}
\begin{proof}
    Consider a diam-mean equicontinuous action $X$ and denote $\pi\colon X\to Y$ for the factor map onto the maximal equicontinuous factor. 
    Since $X$ is mean equicontinuous we know from Proposition \ref{pro:meanEquicontinuousImpliesBanachProximalToMEF} that $\pi$ is Banach proximal. 
    
    In order to show that $\pi$ is regular, we use Theorem \ref{the:regularFactorMapCharacterization} and show that for $y\in Y$ we have $\Diam(\pi^{-1}(y))=0$.     
    Consider $y\in Y$ and $\epsilon>0$. 
    Let $\delta>0$ be such that for all $x\in X$ we have $\Diam(\ball_\delta(x))\leq \epsilon$. 
    From Lemma \ref{lem:meanDiameterFibresVsBallsForBPfactorMap} we observe that 
    $\Diam(\pi^{-1}(y))\leq \sup_{x\in X}\Diam(\ball_\delta(x))\leq\epsilon$. 
    Since $\epsilon>0$ was arbitrary, this shows $\Diam(\pi^{-1}(y))=0$. 
\end{proof}

\subsection{Lifting of diam-mean equicontinuity along regular factor maps}
\label{subsec:liftingDiamMeanEquicontinuityAlongRegular}

\begin{proposition}
    \label{pro:regularLifting}
	Let $\pi\colon X\to Y$ be a regular factor map between actions of $G$. 
	Let $\mathcal{F}=(F_n)$ be a left (or right) F\o lner sequence in $G$. 
	If $y$ is a $\mathcal{F}$-diam-mean equicontinuous point of $Y$ that is $\mathcal{F}$-generic (for some invariant Borel probability measure), then any $x\in \pi^{-1}(y)$ is $\mathcal{F}$-diam-mean equicontinuous in $X$. 
\end{proposition}
\begin{proof}
Consider $x\in \pi^{-1}(y)$. 
Furthermore, denote by $\nu$ the measure for which $y$ is $\mathcal{F}$-generic.
Recall that the notion of a diam-mean equicontinuous point does not depend on the choice of a continuous metric of the respective action. We thus assume w.l.o.g.\ that the metrics on $X$ and $Y$ satisfy $d(x_1,x_2)\geq d(\pi(x_1),\pi(x_2))$ for all $x_1,x_2\in X$. 
Denote $X_0$ for the set of injectivity points of $\pi$ and $Y_0:=\pi(X_0)$. Since $\pi$ is regular, we observe that $\nu(Y_0)=1$. 

Let $\epsilon>0$. 
For $\delta>0$ we denote 
\[G_x^\delta:=\{g\in G;\, \diam(g.\ball_{\delta}(x))>\epsilon\}.\] 
Considering the characterization presented in Proposition \ref{pro:characterizationDiamMeanEquicontinuityViaDensities} it suffices to show that 
there exists $\delta>0$ with
$\upperAdens_\mathcal{F}(G_x^\delta)\leq \epsilon$. 

Since any factor map is closed we know that for $y'\in Y$ there exists 
$\delta_{y'}\in (0,\epsilon/2)$ such that 
$\pi^{-1}(\ball_{2\delta_{y'}}(y'))\subseteq \ball_{\epsilon/2}(\pi^{-1}(y'))$. 
For details see \cite[Theorem~17.7(1)]{alibrantis2006infinite}. 
Exploring the regularity of the Borel measure $\nu$ we choose a compact subset $A\subseteq Y_0$ with $\nu(A)\geq 1-\epsilon/2$.
Furthermore, we choose a finite subset $E\subseteq A$ satisfying 
$A\subseteq \bigcup_{y'\in E}B_{\delta_{y'}}(y')=:O$ 
and denote $\delta':=\min_{y'\in E}\delta_{y'}$. 
Since $y$ is assumed to be $\mathcal{F}$-diam-mean equicontinuous the characterization presented in Proposition \ref{pro:characterizationDiamMeanEquicontinuityViaDensities} allows to observe that there exists $\delta>0$ such that 
\[G_y^\delta:=\{g\in G;\, \diam(g.\ball_{\delta}(y))>\delta'\}\] 
satisfies $\upperAdens_\mathcal{F}(G_y^\delta)\leq \delta'$. 

We will next show that, with this choice of $\delta$, we also have $\upperAdens_\mathcal{F}(G_x^\delta)\leq \epsilon$. This will be achieved by considering
\[G':=\{g\in G;\, g.y\notin O\}\]
and showing (i) $\upperAdens_\mathcal{F}(G'\cup G_y^\delta)\leq \epsilon$ and (ii) $G_x^\delta\subseteq G'\cup G_y^\delta$. 

{(i):}
    Clearly $O=\bigcup_{y'\in E}B_{\delta_{y'}}(y')$ is open and satisfies $\nu(O)\geq \nu(A)\geq 1-\epsilon/2$.
    Thus, $A':=Y\setminus O$ is closed and satisfies $\nu(A')\leq \epsilon/2$. 
    We thus observe from $(F_n)_*\delta_y\to \nu$ that 
    \begin{align*}
        \upperAdens_\mathcal{F}(G')=\limsup_{n\to \infty}\frac{\haar{F_n\cap G'}}{\haar{F_n}}
        =\limsup_{n\to \infty} (F_n)_*\delta_y(A')
        \leq \nu(A')\leq \frac{\epsilon}{2}.    
    \end{align*}
    Since $\delta'=\min_{y'\in E}\delta_{y'}\leq \epsilon/2$ we conclude 
    \begin{align*}
        \upperAdens_\mathcal{F}(G'\cup G_y^\delta)
        \leq \upperAdens_\mathcal{F}(G') + \upperAdens_\mathcal{F}(G_y^\delta)
        \leq \frac{\epsilon}{2} + \delta'
        \leq \epsilon.
    \end{align*}

{(ii):}
    Consider $g\in G_x^\delta$ and assume that $g\notin G'\cup G_y^\delta$. 
    From $g\notin G_y^\delta$ we observe 
    $\diam(g.\ball_\delta(y))\leq \delta'$
    and hence 
    $g.\ball_\delta(\pi(x))=g.\ball_\delta(y)\subseteq \ball_{\delta'}(g.y)$.
    Furthermore, since $g\notin G'$ we have $g.y\in O=\bigcup_{y'\in E}B_{\delta_{y'}}(y')$. 
    Thus, there exists $y'\in E$ with $d(g.y,y')< \delta_{y'}$.
    For such $y'$ we observe 
    $\ball_{\delta'}(g.y)
    \subseteq \ball_{\delta_{y'}}(g.y)
    \subseteq \ball_{2\delta_{y'}}(y')$. 
    Using our particular choice of the metrics on $X$ and $Y$ we compute 
    \[\pi(g.\ball_\delta(x))
    = g.\pi(\ball_\delta(x))
    \subseteq g.\ball_\delta(\pi(x))
    \subseteq \ball_{\delta'}(g.y)
    \subseteq \ball_{2\delta_{y'}}(y') 
    \]
    and hence the choice of $\delta_{y'}$ yields 
    \begin{align*}
        g.\ball_\delta(x)
        \subseteq \pi^{-1}(\pi(g.\ball_\delta(x)))
        \subseteq \pi^{-1}(\ball_{2\delta_{y'}}(y'))
        \subseteq \ball_{\epsilon/2}(\pi^{-1}(y')). 
    \end{align*}
    Since $y'\in E\subseteq A\subseteq Y_0$, we observe that $\pi^{-1}(y')$ is a singleton and hence that $\diam(g.\ball_\delta(x))\leq \epsilon$. 
    This shows $g\notin G_x^\delta$, a contradiction.
\end{proof}

We observe Theorem \ref{the:INTROcharacterizationDMEvsRegular} from Proposition \ref{pro:diamMeanEquicontinuousActionhasRegularMEF} and the following theorem. 

\begin{theorem}
\label{the:liftingDiamMeanEquicontinuityFromEquicontinuousAction}
    Let $\pi\colon X\to Y$ be a regular factor map. 
\begin{itemize}
    \item[(i)] If $Y$ is diam-mean equicontinuous, then $X$ is diam-mean equicontinuous. 
    \item[(ii)] If $Y$ is equicontinuous, then $X$ is diam-mean equicontinuous and $\pi$ is the factor map onto the maximal equicontinuous factor of $X$. 
\end{itemize}
\end{theorem}
\begin{proof}
    Since any equicontinuous action is diam-mean equicontinuous, we obtain (ii) from combining (i) with Proposition \ref{pro:meanEquicontinuousImpliesBanachProximalToMEF}. 

    To show (i) let $x\in X$ and $\mathcal{F}$ be a left F\o lner sequence. 
    From Proposition \ref{pro:meanEquicontinuityResults} we observe that $\pi(x)$ is $\mathcal{F}$-generic. Since $\pi(x)$ is assumed to be diam-mean equicontinuous Proposition \ref{pro:regularLifting} yields that $x$ is $\mathcal{F}$-diam-mean equicontinuous. This shows that $x$ is $\mathcal{F}$-diam-mean equicontinuous w.r.t.\ any left F\o lner sequence in $G$. From Proposition \ref{pro:characterizationDiamMeanEquicontinuityViaAlFolner} we observe that $x$ is diam-mean equicontinuous. 
\end{proof}

\begin{corollary}
    Let $X$ be an action and $\pi$ be the factor map onto the maximal equicontinuous factor. Then $X$ is diam-mean equicontinuous if and only if $\hyper_\pi(X)$ is diam-mean equicontinuous. 
\end{corollary}
\begin{proof}
    Observe that $X$ is conjugated to the subaction $\single(X)\subseteq \hyper_\pi(X)$ and hence the diam-mean equicontinuity of $\hyper_\pi(X)$ implies the diam-mean equicontinuity of $X$. 
    To show the converse assume that $X$ is diam-mean equicontinuous. 
    From Proposition \ref{pro:diamMeanEquicontinuousActionhasRegularMEF} we observe that 
    $\pi\colon X\to X_{\operatorname{eq}}$ is regular and it follows from Theorem \ref{the:regularFactorMapCharacterization} that $\pi_\hyper\colon \hyper_\pi(X)\to X_{\operatorname{eq}}$ is also regular. 
    Since $X_{\operatorname{eq}}$ is equicontinuous we conclude that $\hyper_\pi(X)$ is diam-mean equicontinuous from Theorem \ref{the:liftingDiamMeanEquicontinuityFromEquicontinuousAction}. 
\end{proof}

An action $X$ is called \emph{fully supported} if $X$ equals its maximal support. The maximal support is always a closed invariant and non-empty subset. Thus, any minimal action is fully supported. 
Whenever $X$ is a diam-mean equicontinuous action that is not equicontinuous, then $\single(X)$ is a proper subset of the diam-mean equicontinuous action $\hyper_\pi(X)$, where $\pi\colon X\to X_{\operatorname{eq}}$ denotes the factor map onto the maximal equicontinuous factor of $X$. With the arguments above it follows that $\pi_\hyper\colon \hyper_\pi(X)\to X_{\operatorname{eq}}$ is regular and Theorem $\ref{the:regularFactorMapCharacterization}$ yields that the maximal support of $\hyper_\pi(X)$ is contained in $\single(X)$. We thus observe the following. 

\begin{remark}
    A diam-mean equicontinuous action does not need to be fully supported. 
\end{remark}

Proposition \ref{pro:regularLifting} also allows to say something about regular extensions of $\mathcal{F}$-diam-mean equicontinuous actions. 
An action $X$ is called \emph{pointwise uniquely ergodic} if $\overline{O(x)}$ is uniquely ergodic for any $x\in X$. In particular, any uniquely ergodic action is pointwise uniquely ergodic. 

\begin{proposition}
\label{pro:regularLiftingFdiamME}
    Let $\pi\colon X\to Y$ be a regular factor map and consider a left F\o lner sequence $\mathcal{F}$. 
    Assume that $Y$ is $\mathcal{F}$-diam-mean equicontinuous. 
    \begin{itemize}
        \item[(i)] If $Y$ is fully supported, then $X$ is $\mathcal{F}$-diam-mean equicontinuous. 
        \item[(ii)] If $Y$ is pointwise uniquely ergodic, then $X$ is $\mathcal{F}$-diam-mean equicontinuous. 
    \end{itemize}
\end{proposition}
\begin{proof}
(i): 
    Since $Y$ is $\mathcal{F}$-diam-mean equicontinuous it is $\mathcal{F}$-mean equicontinuous and Proposition \ref{pro:meanEquicontinuityResults} yields that any point is $\mathcal{F}$-generic. We thus observe the statement from Proposition \ref{pro:regularLifting}. 

(ii): 
    Since $Y$ is pointwise uniquely ergodic, any point in $Y$ is $\mathcal{F}$-generic. 
    Thus, the statement follows from Proposition \ref{pro:regularLifting}. 
\end{proof}

\begin{remark}
\label{rem:localBronstein}
    Recall that a factor map $\pi\colon X\to Y$ is called \emph{proximal} if all $x,x'\in X$ with $\pi(x)=\pi(x')$ satisfy $\inf_{g\in G}d(g.x,g.x')=0$. 
    It is straightforward to observe that any action $X$ for which the factor map $\pi_{\operatorname{eq}}$ onto the maximal equicontinuous factor is proximal satisfies the \emph{local Bronstein property} \cite[Definition 2.14]{haupt2025note}.    
    Since any Banach proximal factor map is proximal, we observe from Proposition \ref{pro:meanEquicontinuousImpliesBanachProximalToMEF} that all mean equicontinuous actions satisfy the local Bronstein property.

    Now recall that any diam-mean equicontinuous action is mean equicontinuous and that any regular factor map is Banach proximal (Theorem \ref{the:regularFactorMapCharacterization}). Thus, this observation allows to apply \cite[Theorem 5.2]{haupt2025note} in order to give an alternative proof for the statement of Theorem \ref{the:INTROcharacterizationDMEvsRegular} in the context of minimal actions of countable amenable groups. 
\end{remark}

\section{Stability properties of regular factor maps}
\label{sec:stabilityRegularFactorMaps}

\subsection{Product stability}

For a sequence $(\pi_n\colon X_n\to Y_n)_{n\in \mathbb{N}}$ of factor maps
we consider the factor map $\prod_{n=1}^\infty \pi_n\colon \prod_{n=1}^\infty X_n \to \prod_{n=1}^\infty Y_n$ 
given by
$(x_n)_n\mapsto (\pi_n(x_n))_n$. 
See \cite{cortez2026note} for further reading. 

\begin{proposition}
\label{pro:regularityProductStability}
    Let $(\pi_n\colon X_n\to Y_n)_{n\in \mathbb{N}}$ be a sequence of factor maps. 
    The factor map $\pi:= \prod_{n=1}^\infty \pi_n$ is regular if and only if $(\pi_n)_n$ is a sequence of regular factor maps. 
\end{proposition}
\begin{proof}
We denote $X:= \prod_{n=1}^\infty X_n$ and $Y:= \prod_{n=1}^\infty Y_n$, and write $\pi^{(n)}$ for the projection $X\to X_n$. 

'$\Rightarrow$':      
    Let $m\in \mathbb{N}$. We assume w.l.o.g.\ that the metrics on $X_m$ and $X$ satisfy 
    $d(x_1,x_2)\geq d(\pi^{(m)}(x_1),\pi^{(m)}(x_2))$ for all $x_1,x_2\in X$. 
    This implies 
    $\Diam(\pi^{(m)}(A))\leq \Diam(A)$ for $A\in \hyper(X)$. 
    
    For $y\in Y_m$ we choose a sequence $(y_n)_n\in Y$ with $y_m=y$. Observe that  
    $\pi^{-1}((y_n)_n)=\prod_{n=1}^\infty \pi^{-1}_n(y_n)$ and hence
    $\pi^{(m)}(\pi^{-1}((y_n)_n))=\pi_m^{-1}(y)$. 
    Since $\pi$ is assumed to be regular, we know from Theorem \ref{the:regularFactorMapCharacterization} that 
    \begin{align*}
        \Diam(\pi_m^{-1}(y))
        =\Diam(\pi^{(m)}(\pi^{-1}((y_n)_n)))
        \leq \Diam(\pi^{-1}((y_n)_n))
        =0. 
    \end{align*}
    This shows $\Diam(\pi_m^{-1}(y))=0$ for all $y\in Y_m$ and we observe $\pi_m$ to be regular from another application of Theorem \ref{the:regularFactorMapCharacterization}.

'$\Leftarrow$':
    For $n\in \mathbb{N}$ we denote $X_{n,0}$ for the set of injectivity points of $\pi_n$.
    We observe that $X_0:=\prod_{n=1}^\infty X_{n,0}=\bigcap_{n=1}^\infty (\pi^{(n)})^{-1}(X_{n,0})$ is the set of injectivity points of $\pi$. 
    
    Let $\mu\in \mathcal{M}_G(X)$ and note that $\pi^{(n)}_*\mu\in \mathcal{M}_G(X_n)$ for $n\in \mathbb{N}$.  
    From the regularity of $\pi_n$ we obtain
    \[
    \mu\left(\left(\pi^{(n)}\right)^{-1}(X_{n,0})\right)=(\pi^{(n)}_*\mu)(X_{n,0})=1.
    \]
    Thus, 
    $
    \mu(X_0)
    =\mu\left(\bigcap_{n=1}^\infty \left(\pi^{(n)}\right)^{-1}(X_{n,0})\right)
    =1.$ 
\end{proof}

\subsection{(De-)composition stability}
    The following proposition shows that regularity is stable under composition and decomposition. 

\begin{proposition}
    \label{pro:regularityDeCompositionStability}
    Let $\phi\colon X\to Y$ and $\psi\colon Y\to Z$ be factor maps. Then $\pi:=\psi\circ \phi$ is regular if and only if $\phi$ and $\psi$ are regular. 
\end{proposition}
\begin{proof}
'$\Rightarrow$':
    We assume w.l.o.g.\ that the metrics on $X$ and $Y$ satisfy 
    $d(x_1,x_2)\geq d(\phi(x_1),\phi(x_2))$ for all $x_1,x_2\in X$. 
    For $y\in Y$ and $z\in Z$ we observe from Theorem \ref{the:regularFactorMapCharacterization} that we have 
    $\Diam(\phi^{-1}(y))
    \leq \Diam(\pi^{-1}(\psi(y)))=0$
    and 
    $\Diam(\psi^{-1}(z))
    =\Diam(\phi(\pi^{-1}(z)))
    \leq \Diam(\pi^{-1}(z))=0$. 
    Thus, Theorem \ref{the:regularFactorMapCharacterization} yields that $\phi$ and $\psi$ are regular. 
    
    '$\Leftarrow$': Denote $X_\pi$, $X_\phi$ and $Y_\psi$ for the injectivity points of $\pi$, $\phi$ and $\psi$, respectively. 
    It is a straightforward exercise to verify that 
    $\phi^{-1}(Y_\psi)\cap X_\phi\subseteq X_\pi$. 
    For $\mu\in \mathcal{M}_G(X)$ we have $\phi_*\mu\in \mathcal{M}_G(Y)$ and from the regularity of $\psi$ we observe $\mu(\phi^{-1}(Y_\psi))=\phi_*\mu(Y_\psi)=1$.
    Furthermore, the regularity of $\phi$ yields $\mu(X_\phi)=1$ and we thus have
    $1= \mu(\phi^{-1}(Y_\psi)\cap X_\phi)\leq \mu(X_\pi)\leq 1$, i.e.\ $\mu(X_\pi)=1$. 
\end{proof}

\subsection{$\hyper$-(in)stability}

    Recall that a factor map 
    $\pi \colon X\to Y$
    is called \emph{almost one-to-one} if the set of injectivity points is dense in $X$. 
    In \cite{cortez2026note} it was presented that the properties discussed in the Propositions \ref{pro:regularityProductStability} and \ref{pro:regularityDeCompositionStability} also hold for almost one-to-one factor maps between actions (on compact metric spaces). 
    
    For a factor map $\pi\colon X\to Y$ the image mapping $\hyper(X)\ni A\mapsto \pi(A)$ establishes a factor map $\hyper(\pi)\colon \hyper(X)\to \hyper(Y)$ \cite{cortez2026note}. 
    It was also shown in \cite{cortez2026note, dai2024characterizations} that a factor map $\pi$ between actions (on compact metric spaces) is almost one-to-one if and only if $\hyper(\pi)$ is almost one-to-one. 
    We next present that a similar statement does not hold in the context of regularity. 

\begin{example}
    Let $Y:=\mathbb{Z}\cup \{\infty\}$ be the one-point compactification of $\mathbb{Z}$. 
    We act on $Y$ with $G:=\mathbb{Z}$ by fixing $\infty$ and by defining $g.y:=g+y$ for $g\in G$ and $y\in \mathbb{Z}$. 
    Consider two further disjoint copies of 
    $\hat{X}=\{\hat{n};\, n\in \mathbb{Z}\}\cup \{\hat{\infty}\}$ and $\check{X}=\{\check{n};\, n\in \mathbb{Z}\}\cup \{\check{\infty}\}$ of $Y$ 
    and denote $X':=\hat{X}\cup \check{X}$. 
    We obtain $X$ as a quotient of $X'$ after the identification $\infty_{\hat{X}}=\infty_{\check{X}}$ and denote $\infty$ for the respective element. 
    Note that $X=\{\hat{n};\, n\in \mathbb{Z}\}\cup \{\check{n};\, n\in \mathbb{Z}\}\cup \{{\infty}\}$ is uniquely ergodic and that the invariant Borel probability measure is $\delta_{\infty}$. 
    Thus, the factor map $\pi\colon X\to Y$ given by $\infty\to \infty$ and $\pi(\hat{n}):=\pi(\check{n}):=n$ for $n\in \mathbb{Z}$ is regular. 

    Consider the factor map $\hyper(\pi)\colon \hyper(X)\to \hyper(Y)$, as well as the elements $A:=\{\hat{n};\, n\in \mathbb{Z}\}\cup \{{\infty}\}$ and $B:=\{\check{n};\, n\in \mathbb{Z}\}\cup \{{\infty}\}$ of $\hyper(X)$. 
    Clearly we have $\hyper(\pi)(A)=Y=\hyper(\pi)(B)$ and hence 
    $\{A,B\}\subseteq \hyper(\pi)^{-1}(Y)$. 
    Nevertheless, the elements $A$ and $B$ are fixed in $\hyper(X)$.
    We thus observe $0<d_\hyper(A,B)\leq \Diam(\hyper(\pi)^{-1}(Y))$. 
    From Theorem \ref{the:regularFactorMapCharacterization} we thus conclude that $\hyper(\pi)$ is not regular. 
\end{example}

    Let $\pi\colon X\to Y$ be a factor map and denote $X_0$ and $\hyper(X)_0$ for the sets of injectivity points of $\pi$ and $\hyper(\pi)$, respectively. 
    One readily verifies that
    $
    \single(X_0)=\single(X)\cap \hyper(X)_0.
    $
    From this, the following statement follows directly.
    
\begin{proposition}
    \label{pro:regularityHyperDown}
    Let $\pi\colon X\to Y$ be a factor map.
    Whenever $\hyper(\pi)$ is regular, then $\pi$ is regular. 
\end{proposition}

\begin{remark}
    For an action $X$ (of $G$) the mapping $(g,\mu)\mapsto g_*\mu$ induces an action of $G$ on $\mathcal{M}(X)$.
    Furthermore, for factor map $\pi\colon X\to Y$ the pushforward $\pi_*\colon \mathcal{M}(X)\to \mathcal{M}(Y)$ is a factor map. 
    It was shown in \cite{cortez2026note} that a factor map $\pi$ (between compact metric spaces) is almost one-to-one if and only if $\pi_*$ is almost one-to-one. 
        
    Denote $X_0$ and $\mathcal{M}(X)_0$ for the respective sets of injectivity points of $\pi$ and $\pi_*$. 
    Abbreviating $\delta(X'):=\{\delta_x;\, x\in X'\}$ for $X'\subseteq X$, 
    a straightforward argument yields $\delta(X_0)=\delta(X)\cap \mathcal{M}(X)_0$.
    With a similar argument as above it follows that the regularity of $\pi_*$ implies the regularity of $\pi$. It remains open whether the regularity of $\pi$ is equivalent to the regularity of $\pi_*$. 
\end{remark}

\section{The maximal diam-mean equicontinuous factor}
\label{sec:maximalDiamMeanEquicontinuousFactor}

Next, we establish the existence of a maximal diam-mean equicontinuous factor. 
We begin with the following observation. 

\begin{remark}
    Any subaction of a diam-mean equicontinuous action is diam-mean equicontinuous. 
    Furthermore, the canonical action (of $G$) on a singleton is diam-mean equicontinuous. 
\end{remark}

Diam-mean equicontinuity is stable under (countable) products. 

\begin{theorem}
    Let $(X_n)_{n\in \mathbb{N}}$ be a countable family of diam-mean equicontinuous actions. Then also the induced action on $X:=\prod_n X_n$ is diam-mean equicontinuous. 
    Furthermore, if $Y_n$ is the maximal equicontinuous factor of $X_n$, then $Y:=\prod_n Y_n$ is the maximal equicontinuous factor of $X$. 
\end{theorem}
\begin{proof}
    For $n\in \mathbb{N}$ denote $\pi_n\colon X_n\to Y_n$ for the respective factor map onto the maximal equicontinuous factor of $X_n$. 
    Note that $Y$ is equicontinuous \cite[Chapter 2]{auslander1988minimal}. 
    From Proposition \ref{pro:diamMeanEquicontinuousActionhasRegularMEF} we observe that all $\pi_n$ are regular. 
    Thus, Proposition \ref{pro:regularityProductStability} yields that the induced factor map $X\to Y$ is also regular and we have established that $X$ is a regular extension of an equicontinuous action. The statement now follows from Theorem \ref{the:liftingDiamMeanEquicontinuityFromEquicontinuousAction}. 
\end{proof}

It is well known that a property (P), which holds for the canonical singleton action and which is stable under products and subactions allows for the construction of a 'maximal factor of $X$ that satisfies (P)' for any action $X$. 
For this, it is often referred to \cite{auslander1988minimal}, where this technique is presented in the context of actions on compact Hausdorff spaces and product-stability means stability of (P) under arbitrary (possibly uncountable) products. 
However, working with compact metric spaces one can only consider countable products and the question arises, whether the technique of \cite{auslander1988minimal} can also be used while working only with actions on compact metric spaces. 
We will present in the proposition below that it can be used and provide a full proof for the convenience of the reader. This allows to establish Theorem \ref{the:INTROmaximalDiamMeanEquicontinuousFactor}. 
For a precise terminology, we call two factor maps $\pi\colon X\to Y$, $\pi'\colon X\to Y'$ \emph{conjugated} if there exists a conjugacy $\iota\colon Y\to Y'$ with $\iota\circ \pi=\pi'$. 

\begin{proposition}
\label{pro:auslanderRefined}
    Let $G$ be a topological group. 
    Consider a property (P) that is defined for actions of $G$ on compact metric spaces such that
    \begin{itemize}
        \item[(i)] Whenever an action satisfies (P), then any subaction satisfies (P), 
        \item[(ii)] Whenever $(X_n)_{n\in \mathbb{N}}$ is a countable family of actions that satisfy (P), then $\prod_n X_n$ satisfies (P), and
        \item[(iii)] The canonical singleton action of $G$ satisfies (P). 
    \end{itemize}
    If $X$ is an action of $G$ (on a compact metric space), then there exists a maximal factor of $X$ that satisfies (P), i.e.\ a factor map $\pi_{P}\colon X\to X_P$ that satisfies (P), such that for all factors $\pi\colon X\to Y$ that satisfy (P) there exists a factor map $\psi\colon X_P\to Y$ with $\pi=\psi\circ \pi_{P}$. 
    The maximal factor of $X$ that satisfies (P) is unique up to conjugacy. 
\end{proposition}

\begin{remark}
    Consider properties (E) and (D) of actions such that (E) is stronger than (D), i.e.\ any action that satisfies (E) also satisfies (D). 
    Whenever (E) and (D) satisfy the properties (i), (ii) and (iii), then for an action $X$ we can consider the maximal factor $\pi_D\colon X\to X_D$ that satisfies (D). 
    It is easy to verify that the maximal factors $\pi\colon X\to X_E$ and $\phi\colon X_D\to X_E'$ that satisfy (E) are conjugate, i.e.\ there exists a conjugacy $\iota\colon X_E'\to X_E$ such that 
    $\pi=\iota\circ \phi\circ \pi_D$. 
    In particular, we observe that the maximal equicontinuous factor of an action $X$ is the maximal equicontinuous factor of the maximal diam-mean equicontinuous factor of $X$. 
\end{remark}

\begin{remark}
    The maximal equicontinuous factor and the maximal diam-mean equicontinuous factor differ in general. For example any regular Toeplitz subshift of $\{0,1\}^\mathbb{Z}$ is a (non-trivial) regular extension of an odometer and hence diam-mean equicontinuous. Here the odometer is the maximal equicontinuous factor. See \cite{downarowicz2005survey, baake2016toeplitz, cortez2014invariant} for details and further reading. 
\end{remark}

Before presenting the proof, we briefly discuss some obstacles that arise in this context.

\begin{example}
    It is a well-known folklore result that the class of all compact metric spaces, up to homeomorphism, is uncountable.
    Let $X$ be the Cantor space and $G$ be any group. We act on $X$ by fixing all elements. Since all compact metric spaces are continuous images of $X$ \cite[Theorem 30.7]{stephen2004general} we observe that the family of all factors (up to conjugacy) is uncountable.
\end{example}

\begin{proof}[Proof of Proposition \ref{pro:auslanderRefined}:]
    Note that the concepts of factor map, conjugacy, invariant subset and product given in Section \ref{sec:prelims} can easily be extended to actions of $G$ on a compact Hausdorff space. 
    In particular, this allows to consider arbitrary products of actions. See \cite{auslander1988minimal} for further reference in this direction. 

    Let $X$ be an action of $G$ on a compact metric space. 
    For simplicity, a factor of $X$ is called a \emph{P-factor} if it satisfies (P). 
    For any factor $\pi$ of $X$ we denote $R(\pi):=\{x,x'\in X;\, \pi(x)=\pi(x')\}$. 
    Note that two factors $\pi$ and $\pi'$ of $X$ are conjugated if and only if $R(\pi)=R(\pi')$. 
    It follows that a family $(\pi_i\colon X\to Y_i)_{i\in I}$ of conjugacy-representatives of P-factors of $X$ can be treated as a set. From (iii) we know that $I$ is non-empty.    
    Consider the product action $Y_I:=\prod_{i\in I} Y_i$ and the map 
    $\pi_I\colon X\to Y_I$ 
    given by 
    $\pi_I(x):=(\pi_i(x))_i$. 
    Clearly, $\pi_I$ restricts to a factor map 
    $\pi_{P}\colon X\to X_P:=\pi_I(X)$. 
    It is straightforward to verify that for any P-factor $\pi\colon X\to Y$ 
    there exists a factor 
    $\psi\colon X_P\to Y$ 
    with 
    $\pi=\psi\circ \pi_{P}$. 
    Furthermore, from $R(\pi_{P})=\bigcap_i R(\pi_i)$ it is straightforward to observe that $X_P$ is conjugated to any other candidate for a maximal P-factor. 
    This establishes the uniqueness of the maximal P-factor of $X$ up to conjugacy.

    It remains to show that $X_P$ satisfies (P). This does not immediately follow from (ii), since $I$ is not necessarily countable.     
    However, since $X$ is metrizable we observe that $X_P$ is metrizable. Thus, $X_P$ has a countable base $(B^{(n)})_{n\in \mathbb{N}}$ for its topology. Revising the definition of the product topology, we assume w.l.o.g.\ that $B^{(n)}$ is of the form 
    $B^{(n)}=X_P\cap (\prod_i U_i^{(n)})$ 
    with $U_i^{(n)}\subseteq Y_i$ open and non-empty 
    and $U_i^{(n)}=Y_i$ for all but finitely many $i$. 
    
    Let $N$ be the set of all $i\in I$, for which there exists some $n\in \mathbb{N}$ such that $U_i^{(n)}\neq Y_i$ and note that $N$ is countable.
    Denote $Y_N:=\prod_{i\in N} Y_i$ and consider the projection 
    $\pi_{I,N}\colon Y_I\to Y_N$. 
    Let $\iota$ be the restriction $\iota\colon X_P\to \iota(X_{P})$ of $\pi_{I,N}$
    and observe that $\iota$ is a factor map. 
    Our choice of $N$ implies that the topology of $X_P$ is given by $\{\pi_{I,N}^{-1}(O);\, O\in \tau\}$, where $\tau$ is the topology of $\iota(X_{P})$. 
    
    Since $\iota(X_{P})$ is Hausdorff, we observe $\iota$ to be injective and hence a conjugacy.     
    Furthermore, from (ii) we know that $Y_N$ satisfies (P). 
    Thus, $\iota(X_{P})\subseteq Y_N$ yields that also $\iota(X_{P})$ and hence $X_P$ satisfy (P).  
\end{proof}

\bibliographystyle{alpha}
\bibliography{ref}

\end{document}